\documentclass[a4paper,11pt]{amsart}
\usepackage[latin1]{inputenc}
\usepackage[T1]{fontenc}
\usepackage[english]{babel}
\usepackage{amssymb}
\usepackage{amsmath}
\usepackage{amsthm}
\usepackage{amscd}
\usepackage{amsfonts}
\usepackage{stmaryrd}
\usepackage{pb-diagram}
\usepackage{a4wide}
\usepackage{nextpage}
\usepackage{fancyhdr}
\usepackage{hyperref}
\usepackage{comment} 
\usepackage{mathtools}
\DeclarePairedDelimiterX\set[1]\lbrace\rbrace{\def\given{\;\delimsize\vert\;}#1}

\usepackage[dvipsnames]{xcolor}
\newcommand{\bleu}[1]{\textcolor{blue}{#1}}
\definecolor{cadmiumgreen}{rgb}{0.0, 0.42, 0.24}

\newtheorem{prop}{Proposition}[section] 
\newtheorem{cor}[prop]{Corollary} 

\newtheorem{lem}[prop]{Lemma}

\newtheorem{theorem}[prop]{Theorem}

\theoremstyle{definition}
\newtheorem{rem}[prop]{Remark}

\newcommand{\Card} {\mathop{\mathrm{Card}}}

\DeclareMathOperator{\GL}{\mathrm{GL}}
\newcommand{\SL} {\mathop{\mathrm{SL}}}

\newcommand{\cl} {\mathop{\mathrm{Cl}}}

\newcommand{\Gal} {\mathop{\mathrm{Gal}}}

\newcommand{\End} {\mathop{\mathrm{End}}}
\DeclareMathOperator{\sgn}{\mathrm{sgn}}

\newcommand{\BC} {\mathbb{C}}
\newcommand{\BF} {\mathbb{F}}

\newcommand{\BR} {\mathbb{R}}
\newcommand{\BZ} {\mathbb{Z}}
\newcommand{\BN} {\mathbb{N}}

\newcommand{\hTag}{h_{\mathrm{Tag}}}

\renewcommand{\ker}{\mathop{\mathrm{Ker}}}
\newcommand{\ie}{\emph{i.e.}}

\begin{document}

\title{Drinfeld singular moduli, hyperbolas, units}
\author{Bruno Angl\`es, C\'ecile Armana, Vincent Bosser, Fabien Pazuki}

\address{Laboratoire de Math\'ematiques Nicolas Oresme,
Universit\'e de Caen-Normandie, BP 5186
14032 Caen, France}
\email{bruno.angles@unicaen.fr}

\address{Universit\'e de Franche Comt\'e, CNRS, LmB (UMR 6623), F-25000 Besan\c{c}on, France}
\email{cecile.armana@univ-fcomte.fr}

\address{Laboratoire de Math\'ematiques Nicolas Oresme,
Universit\'e de Caen-Normandie, BP 5186
14032 Caen, France}
\email{vincent.bosser@unicaen.fr}

\address{Department of Mathematical Sciences, University of Copenhagen,
Universitetsparken 5, 
2100 Copenhagen \O, Denmark}
\email{fpazuki@math.ku.dk}

\thanks{The authors were supported by the IRN MaDeF (CNRS), IRN GandA (CNRS), and by ANR-20-CE40-0003 Jinvariant. They also benefitted from support of Laboratoire de Math\'ematiques Nicolas Oresme (Caen) and Laboratoire de Math\'ematiques de Besan\c{c}on.
The final steps of this work where performed while the last author was in residence at Institut Mittag-Leffler in Djursholm, Sweden, during the semester in Analytic Number Theory: tack s\aa{} mycket! We thank Tim Trudgian for interesting feedback. We dedicate this work to Bas Edixhoven and to Andrea Surroca, who passed away way too early.} 

\maketitle
\vspace{-0.9cm}
\begin{center}

\vspace{0.2cm}
\today
\end{center}

\begin{abstract}
Let $q\geq2$ be a prime power and consider Drinfeld modules of rank 2 over $\mathbb{F}_q[T]$. We prove that there are no points with coordinates being Drinfeld singular moduli, on a family of hyperbolas $XY=\gamma$, where $\gamma$ is a polynomial of small degree. This is an effective Andr\'e-Oort theorem for these curves. We also prove that there are at most finitely many Drinfeld singular moduli that are algebraic units, for every fixed $q\geq2$, and we give an effective bound on the discriminant of such singular moduli. We give in an appendix an inseparability criterion for values of some classical modular forms, generalising an argument used in the proof of our first result.
\end{abstract}

{\flushleft
\textbf{Keywords:} Drinfeld singular moduli, Andr\'e-Oort, Diophantine equations.\\
\textbf{Mathematics Subject Classification:} 11G09, 11J93. }

\thispagestyle{empty}

\section{Introduction}

In \cite{Andre}, Andr\'e proves that an irreducible plane curve that is not a vertical line, not a horizontal line, and not a modular curve, contains at most finitely many points with coordinates $(j_1,j_2)$, where $j_1$ and $j_2$ are \textit{singular moduli}, \ie{} they are $j$-invariants of CM elliptic curves $E_1$ and $E_2$.

Together with work of Edixhoven \cite{Ed98} and work of Oort \cite{Oo97}, this led to a vast generalisation called the Andr\'e-Oort conjecture, which opened deep questions on Shimura varieties, where many new ideas have amazed mathematicians in recent years.



This first result of Andr\'e was made effective by K\"uhne in \cite{Kuhne}, where his Theorem 2 gives an explicit bound on the discriminant of the endomorphism rings associated to the CM elliptic curves $E_1$ and $E_2$.

In \cite{BMZ}, Bilu, Masser, and Zannier publish their own effective proof of Andr\'e's result, and apply in particular their method to the explicit example of the hyperbola $XY=1$, where not only one obtains finiteness: it turns out that there is no solution in singular moduli at all.

Motivated by this example, Masser asked the question of whether a singular modulus, which is always an algebraic integer, can be a unit. Indeed, singular moduli which would be solutions to $XY=1$ would in particular be units. Habegger \cite{Hab15} first answered that there is at most finitely many singular moduli which are also units, which was generalised by Li \cite{Li21} to values of modular polynomials at singular moduli. Bilu, Habegger, K\"uhne \cite{BHK} finally proved that there is no singular modulus which is also a unit. This was generalised by Campagna in \cite{Cam21} to $S$-unit singular moduli, when $S$ is the infinite set of primes congruent to $1$ modulo~$3$.

Inspired by the classical analogy between elliptic curves and Drinfeld modules of rank 2, we are interested in counterparts of these results in the Drinfeld setting. The equivalent of Andr\'e's result is due to Breuer, see \cite{Bre05, Bre07, Bre12}. 

We consider Drinfeld modules over $A=\mathbb{F}_q[T]$, where $q$ is a prime power. If $\Phi$ is a Drinfeld module and $\Lambda\subset\BC_{\infty}$ is the associated lattice, we identify $\End_{\BC_{\infty}}(\Phi)$ with the subring
$\{z\in\BC_{\infty}\mid z\Lambda\subset\Lambda\}$ of $\BC_{\infty}$. A Drinfeld module $\Phi$ of rank $2$ is said to have complex multiplication (or to be CM for short)
if $\End_{\BC_{\infty}}(\Phi)\not=A$. The ring $\End_{\BC_{\infty}}(\Phi)$ is then a free $A$-module of rank $2$.
We will say that $\alpha\in\BC_{\infty}$ is a \emph{singular modulus} if there exists a CM Drinfeld module $\Phi$ of rank $2$ such that $j(\Phi)=\alpha$. It is known that singular moduli are integral over $\BF_q[T]$. 

In our first result, we provide an effective solution to the Andr\'e-Oort problem for the family of curves given by the equations $XY=\gamma$ for a polynomial $\gamma$ with degree less than or equal to $q^2-2$: points with coordinates that are singular moduli do not exist on these curves! Let $\mathbb J_{CM}= \{ j(\Phi)\mid $ $\Phi $ is a CM $A$-Drinfeld module of rank two$\}.$ We prove the following theorem:

\begin{theorem}\label{Effect1-oddchar-intro}  Let $\gamma\in \overline{\mathbb  F}_q[T]\setminus \{0\}$. 
Assume $\deg \gamma\leq q^2-2.$
Then the equation $j_1 \, j_2=\gamma$ has no solutions for $j_1,j_2\in \mathbb J_{CM}.$
\end{theorem}

This implies in particular that there are no Drinfeld singular moduli satisfying the equation $XY=1$. The proof is partly based on the work of Brown \cite{Brown} and on the work of Hsia and Yu \cite{HsYu98}.
Note that we also give a proof in characteristic~2. The method is different from what is existing in the elliptic curve case, and more intrinsic to the function field setting. We also give an explicit example (see Remark \ref{Hayes8}) to show that the bound $q^2-2$ on the degree of $\gamma$ is sharp in general. An interesting technical step in characteristic $2$, that is used in our proof, is an inseparability criterion of values of certain Drinfeld modular forms. What is necessary to obtain the main results here is the CM case, which is due to Schweizer \cite{Sch97}, but as we found a way to generalise the result to non-CM values, we included it as an appendix.

We then move to the related Masser question. Drinfeld singular moduli satisfying $XY=1$ would actually be units. We do not rule out the existence of Drinfeld singular moduli that are units, nevertheless we prove that for each fixed $q\geq 2$, there are at most finitely many of them, and we give an upper bound on their discriminant. Here is the statement:

\begin{theorem}\label{mainth}
Let $q\geq 2$ be a prime power. There are at most finitely many singular moduli that are algebraic units over $\BF_q(T)$.

More precisely, if such a singular modulus exists, then it must have complex multiplication by an order with discriminant $D$ satisfying 
$$
\log_q^+ \log_q \sqrt{\vert D \vert} \leq  \frac{2400 (q+1)}{\ln 2},
$$
where for $x\in{\BR}$ we set $\log_q^+(x)=\log_q(\max\{1,x\})$.
\end{theorem}

The idea of the proof stems from the strategy of Bilu, Habegger, K\"uhne, in the case of elliptic curves. Ran \cite{ZR} also independently announced the finiteness of Drinfeld singular moduli that are units, under some hypothesis. Our proof is unconditional, in particular includes the case of characteristic $2$, which required additional methods. All the constants in our bounds are given explicitly, which leads to the effective control in Theorem \ref{mainth}. The dependence in $q$, however, does not seem to indicate that the finiteness result in Theorem \ref{mainth} could be uniform in $q$: this remains an open question.

Let us describe the strategy of proof: the main idea is to bound from above the discriminant of the endomorphism ring of such Drinfeld modules $\Phi$. To reach this goal, we use height theory and produce an upper bound on the height of $\Phi$, given in Corollary~\ref{maj-hjb}, and no less than two lower bounds on the height of $\Phi$, given in Proposition \ref{easybound} and in Corollary \ref{weibound}. These results are compared in the final step of the proof, where the tension between these bounds leads to an upper bound on the discriminant, which we show is enough to conclude. The effectivity in Theorem \ref{mainth} relies on the powerful inequality (\ref{laclef}).

\section{Preliminaries and basic properties}

This section introduces the notation and the preliminary results that will be needed for the proof of Theorems
\ref{Effect1-oddchar-intro} and \ref{mainth}.

In the whole text, we denote by $A=\BF_q[T]$ the polynomial ring in the indeterminate $T$ over the finite field $\BF_q$, by
$k=\BF_q(T)$ its quotient field, by $k_{\infty}=\BF_q((1/T))$ the completion of $k$ at the infinite place $\infty=1/T$,
and by $\BC_{\infty}$ the completion of an algebraic closure of $k_{\infty}$.
We denote by $\vert . \vert$ the absolute value on $\BC_{\infty}$ normalized by $\vert T \vert=q$, and
for $z\in\BC_{\infty}^{\times}$ we set $\deg z = \log_q\vert z\vert$, where $\log_q$ is the logarithm to the base $q$. Thus, for any non zero $a\in A$, $\deg a$ is the degree of the polynomial $a$. We use the symbol $\ln$ to denote the logarithm to the base $e$, i.e. satisfying $\ln e =1$. We denote by $\mathbb{N}=\{0,1,\ldots\}$ the set of non-negative natural integers.

If ${\mathcal K}$ is a subfield of $\BC_\infty$, we denote by $\overline{{\mathcal K}}$ its algebraic closure in $\BC_\infty$, and by ${\mathcal K}^{sep}$ its separable closure in $\overline{{\mathcal K}}$.
It will be convenient to introduce the function $\sgn: \overline{{\BF}}_q((1/T))^{\times}\rightarrow {\BF}_q^{\times}$,
defined as follows: for $x\in \overline{{\BF}}_q((1/T))$, $\sgn(x)$ is the first non zero coefficient of the $1/T$-expansion of $x$.
In particular, if $a\in A$ is non zero, $\sgn(a)$ is nothing else than the leading coefficient of the polynomial $a\in \BF_q[T]$.
We will also denote by $A_+$ the set of non zero polynomials $a\in A$ such that $\sgn(a)=1$ (i.e., monic polynomials).

\subsection{Drinfeld modules and Drinfeld modular forms}\label{Drinfeld intro}

For this paragraph, we refer the reader to classical references such as \cite{GekLNM}, \cite{Ge}, \cite{Goss}, \cite{Pap}.

Let $\Phi$ be an $A$-Drinfeld module of rank $2$ over $\BC_{\infty}$. It is given by a twisted polynomial of the form
$$ \Phi_T = T + g \tau + \Delta \tau^2 \in \BC_{\infty}\{ \tau \}$$
where $g$ and $\Delta$ are in $\BC_{\infty}$ and $\Delta \neq 0$. The $j$-invariant of $\Phi$ is defined as
$$ j(\Phi) = \frac{g^{q+1}}{\Delta}.$$
It characterises the isomorphism class of the Drinfeld module $\Phi$ over $\BC_{\infty}$.

Let $\Omega=\BC_{\infty}\setminus k_{\infty}$ be the Drinfeld upper half-plane.
For $z\in\BC_{\infty}$, define
$$\vert z\vert_i=\displaystyle \inf_{x\in k_{\infty}}\vert z-x\vert=\min_{x\in k_{\infty}}\vert z-x\vert
\qquad \mbox{\rm and}\qquad
\vert z\vert_A=\displaystyle \inf_{a\in A}\vert z-a\vert=\min_{a\in A}\vert z-a\vert.
$$

We will denote by $\mathcal F$ the fundamental domain for the action of $\GL(2,A)$ on $\Omega$, that is:
$$
{\mathcal F}=\{z\in\Omega\mid \vert z \vert=\vert z \vert_i\ge 1\}.
$$
Note that for $z\in\mathcal F$, we have $\vert z\vert_i=\vert z\vert_A$.

To any $z\in \Omega$, one can attach the rank $2$ lattice $\Lambda_z = A \oplus A z$ in $\BC_{\infty}$ as well as a rank~$2$ Drinfeld module $\Phi^{(z)} : A  \to \BC_{\infty} \{ \tau \}$ such
that $\ker \exp_{\Phi^{(z)}}= \Lambda_z$ where $\exp_{\Phi^{(z)}}$ is the exponential function attached to the Drinfeld module. The Drinfeld module $\Phi^{(z)}$ associated to $\Lambda_z$ is given by:
$$\Phi^{(z)}_T= T+g(z)\tau+\Delta(z) \tau^2 \in \BC_{\infty} \{ \tau \}$$
where $g(z)$ and $\Delta(z)$ are in $\BC_{\infty}$ and $\Delta(z)\neq 0$.
The $j$-invariant function is now defined as:
$$\forall z \in \Omega,\quad j(z)=j(\Phi^{(z)}) =\frac{g(z)^{q+1}}{\Delta(z)}.$$
It is known that $g$ and $\Delta$ are Drinfeld modular forms for $\GL_2(A)$ of weight $q-1$ and $q^2-1$, respectively (see \cite{Ge}). Also the $j$-invariant is a Drinfeld modular function for $\GL_2(A)$ of weight $0$.

We know that the zeros of the $j$-invariant inside the fundamental domain $\mathcal F$ are exactly the elements of $\BF_{q^2}\setminus\BF_q$ and that they are of order $q+1$ (see
\cite[(5.15)]{Ge}). We will denote by ${\mathcal E}=\BF_{q^2}\setminus\BF_q$ the set of these zeros (these are the so-called \emph{elliptic points} of $\Omega$).

\subsection{Complex multiplication}
\label{complexmultiplication}

To prove Theorem~\ref{mainth} we will need results on complex multiplication that we present here.


Let $\Phi$ be a Drinfeld $A$-module of rank $2$, and let $\Lambda_{\Phi}\subset\BC_{\infty}$ be its associated lattice. We identify
the endomorphism ring $\End_{\BC_{\infty}}(\Phi)$ with the subring
$\{z\in\BC_{\infty}\mid z\Lambda_{\Phi}\subset\Lambda_{\Phi}\}$ of $\BC_{\infty}$. The Drinfeld module $\Phi$ is said to have complex multiplication
(or to be CM for short) if $\End_{\BC_{\infty}}(\Phi)\not=A$. The ring $\End_{\BC_{\infty}}(\Phi)$ is then a free $A$-module of rank $2$.
We will say that a point $z\in\Omega$ is CM if the associated Drinfeld module $\Phi^{(z)}$ is CM.
Finally, we say that $\alpha\in\BC_{\infty}$ is a \emph{singular modulus} if there exists a CM Drinfeld module $\Phi$ of rank $2$ such that $j(\Phi)=\alpha$. Thus, $\alpha$ is a singular modulus if and only if $\alpha$ is of the form $\alpha=j(z)$ for some CM point $z\in\Omega$ (then all points $z\in\Omega$ such that $\alpha=j(z)$
are CM points since they correspond to isomorphic Drinfeld modules).
We will denote by $\mathbb J_{CM}$ the set of all singular moduli.

Now, let $K/k$ be a quadratic extension of $k$ with $K\subset \BC_{\infty}$.
This extension is said to be \emph{imaginary} if there is no embedding of $K$ in $k_{\infty}$, or
equivalently if the place $\infty=1/T$ is not split in $K$. If $K=k(z)$, this means that $z\notin k_{\infty}$.
If $K/k$ is an imaginary quadratic extension, we will denote by ${\mathcal O}_K$ the ring of integers of $K$.
Recall that an \emph{order} in $K$ is an $A$-subalgebra of ${\mathcal O}_K$ whose field of fractions is equal to $K$: It is then
a free $A$-submodule of ${\mathcal O}_K$ of rank $2$.
For such an order $\mathcal O$ we will denote by $\cl({\mathcal O})$ its ideal class group 
and by $h(\mathcal O)=\Card(\cl({\mathcal O}))$ its class number, see e.g. \cite{Ros}, Chapter 17, Definition on page 315.

We have the following result, which is valid in every characteristic:

\begin{lem}\label{CM1}
Let $z\in \Omega$ be a CM point. Then:

(i) The field $K=k(z)$ is an imaginary quadratic field and ${\mathcal O}_z=\End(\Phi^z)$ is an order in~$K$;

(ii) The element $j(z)$ is integral over $A$;

(iii) The extension $K(j(z))/K$ is abelian with $\Gal(K(j(z))/K)\simeq \cl({\mathcal O}_z)$;

(iv) If $K/k$ is separable, then $k(j(z))/k$ is separable and $[k(j(z)):k]=[K(j(z)):K]$. 

\end{lem}

\begin{proof} This is a consequence of \cite[Section 4]{Gek83} and \cite[Theorem~3]{Sch97}.
\end{proof}

If $\alpha = j(z)$ is a singular modulus, the field $K=k(z)$ and the ring $\mathcal{O}_z$ depend only on the class of $z$ modulo $\mathrm{GL}_2(A)$ (hence on $\alpha$ only), and will be called respectively the CM field and CM ring of $\alpha$.

Note that when $q$ is odd, every quadratic extension $K$ of $k$ in separable. When $q$ is even, there is only one
inseparable quadratic extension of $k$, namely $K=\BF_q(\sqrt{T})$. When the extension $K/k$ is separable,
one can define the notion of discriminant of an order ${\mathcal O}$ in $K$,
see e.g. \cite[Section~2]{PC}: This is a non zero element of $A$ 
which is actually defined only up to a square in $\BF_q^{\times}$.
There is no canonical choice for the discriminant in odd characteristic, but in characteristic two the discriminant
can be chosen monic (since every element of $\BF_q^{\times}$ is a square). In the sequel, we will denote by $D_{\mathcal O}$ any
choice of a discriminant for ${\mathcal O}$ when the characteristic is odd, and the unique monic discriminant of ${\mathcal O}$
when the characteristic is even.
When ${\mathcal O}={\mathcal O}_K$, the discriminant $D_{{\mathcal O}_K}$ is the discriminant of $K$ and will simply be
denoted by $D_K$.

As in zero characteristic (see e.g. \cite[\S 7]{Cox}), given an imaginary quadratic extension $K/k$ and an
order ${\mathcal O}\subset {\mathcal O}_K$, one can define its \emph{conductor} as 
for instance the unique element $f\in A_+$ such that $\displaystyle {\mathcal O}_K/{\mathcal O}\cong A/fA$. We then have
${\mathcal O}=A+f{\mathcal O}_K$ and, if $K/k$ is separable, $D_{\mathcal O}=f^2 D_K$.

The theory of complex multiplication now splits into two cases: the case of odd characteristic and the case of characteristic $2$. We will start with the odd case and deal with the even case afterwards.


\subsubsection{Case $q$ odd}

\

In this section, we assume that $q$ is odd. In that case, every imaginary quadratic extension $K$ of $k$ is separable
and has the form $K=k(\sqrt{D})$,
where $D\in A$ is 
not a square in $k$. It is easy to see that if a non zero element $D \in A$ is given, 
then the field $K=k(\sqrt{D})$ is an imaginary quadratic extension of $k$ if and only if the following
condition holds:
\begin{equation}
\begin{cases}\label{cond-imaginary}
\text{(i) $\deg D$ is odd}\\
\text{or}  \\
\text{(ii) $\deg D$ is even and the leading coefficient of $D$ is not a square in $\BF_q^{\times}$}.
\end{cases}
\end{equation}
Note that Conditions (\ref{cond-imaginary}) (i) and (ii) are also equivalent to the facts that $\infty=1/T$ is ramified and
inert in $K/k$, respectively, see Proposition 14.6 in \cite{Ros}.
When Condition (\ref{cond-imaginary}) holds, the $A$-module
$A[\sqrt{D}]=A+A\sqrt{D}$ is an order in $K$ whose discriminant is $D$, and this is actually the unique order in $K$ whose
discriminant is $D$. In particular, we have $\mathcal{O}_K=A[\sqrt{D_K}]$.
For any fixed choice of a square root $\sqrt{D}$ in $\overline{k}$, let us denote by $S_D$ the following set of elements of
$k(\sqrt{D})$: 
$$S_D = \left\{ \frac{-b+\sqrt{D}}{2a} \mid a\in A_+,\, b\in A;\, \exists c\in A, b^2-4ac=D,\,
\vert b \vert < \vert a\vert \leq \vert c\vert,\, \gcd(a,b,c)=1 \right\}.$$

Note that although the discriminant of an order is defined only up to a square in $\BF_q$, 
we require here that $b^2-4ac$ is \emph{equal} to $D$.

We will need the following result:

\begin{lem}\label{Ord}${}$
Let $D$ be an element in $A$ such that $K=k(\sqrt{D})$ is an imaginary  quadratic extension of $k$. Denote by ${\mathbb J}_{CM}(D)$ the set of singular moduli having complex multiplication by the order of $K$ of discriminant $D$, and let $z\in  S_D$. Then:
\begin{enumerate}
    \item\label{Ord-1} $|z|=| z|_i=| z|_A\geq 1$.
    \item\label{Ord-2} 
    $\displaystyle {\mathbb J}_{CM}(D)=\{ \sigma (j(z))\mid \sigma \in {\rm Gal}(K^{sep}/K)\}=j(S_D).$
    \item\label{Ord-3} If $\alpha_1,\ldots,\alpha_m$ denote the conjugates of $j(z)$ over $k$, we have:
    $${\mathbb J}_{CM}(D)=\{\alpha_1,\ldots,\alpha_m\}.$$
\end{enumerate}
\end{lem}

\begin{proof}
Part (1) follows for example from \cite[Lemma 12]{PC}. 
Part (2) and Part (3) are a consequence of \cite {Gek83}, Section~4, from Lemma~\ref{CM1}(iv) and direct computations. See also the proof of Theorem 7.5.21 of \cite{Pap}.
\end{proof}

\subsubsection{Case $q$ even}\label{CMcar2}
In this section we assume that $q$ is a power of two. In that case, the situation is different depending on whether the
quadratic extension is separable or inseparable over $k$.

Let us first treat the separable case. In that case, the basic facts on separable quadratic extensions of $k$ are due to
Hasse \cite{Hasse} and Chen \cite{Che}.

Let $K/k$ be a separable imaginary extension of degree two. Then by \cite{Hasse} (see also \cite{Che}), there exists an element
$\xi\in\overline{k}$ such that $K=k(\xi)$ and which satisfies an equation of the form
\begin{equation}\label{xi}
\xi^2+\xi =\frac{B}{C},
\end{equation}
where $B$ and $C$ are elements in $A$ that can be chosen
such that the following conditions hold:
\begin{enumerate}
    \item[(i)] $C\in A_+$, $\gcd(B,C)=1$ and $\deg B\geq \deg C$;
    \item[(ii)] The prime factorization of $C$ has the form $C=\prod_{i=1}^{r}P_i^{e_i}$, where $P_1, \ldots, P_r$ are distinct monic
irreducible polynomials and $e_i\equiv 1\pmod{2}$ for all $i\in\{1,\ldots,r\}$;
    \item[(iii)]
    \begin{enumerate}
        \item\label{ramified} If $\deg B>\deg C$ then $\deg B-\deg C \equiv 1 \pmod{2}$;
        \item\label{inert} If $\deg B=\deg C$ then the polynomial $X^2+X+\sgn(B)$ is irreducible in $\mathbb F_q[X]$.
    \end{enumerate}
\end{enumerate}
Note that $C=1$ is allowed here and satisfies Condition (ii) with $r=0$ by convention (empty product).
In the case (iii.a) the place $\infty$ is ramified in $K/k$, and in the case (iii.b) the place $\infty$ is inert.
For $i=1, \ldots, r,$ let $n_i$ be the smallest integer such that $n_i\geq \frac{e_i}{2}$ and let $G= \prod_{i=1}^r P_i^{n_i}$.
Then we have, by \cite{Che}:
  $$\mathcal{O}_K= A\oplus AG\xi\quad \mbox{\rm and}\quad D_K=G^2. $$
 
It is not difficult to deduce from this (e.g. as in \cite{Cox}) that for any given $f\in A_+$,
there exists a unique order $\mathcal{O}$ in $K$ whose conductor is $f$, namely, $\mathcal{O}=A\oplus AfG\xi$.
The discriminant of $\mathcal{O}$ is then $D_{\mathcal{O}}=f^2G^2=f^2D_K$.
Note that contrary to the odd characteristic case, the discriminant does not determine
the quadratic field $K$: In characteristic two, there are several quadratic fields and orders with the same
discriminant. However, if the field $K$ is fixed (e.g. by the choice of an element $\xi$ as above), an order in $K$ is uniquely determined by its discriminant. Note also that the discriminant of an order $\mathcal{O}$ is here always monic and a square in $A$.

We will further set $\mathrm{rad}(G)=P_1\cdots P_r,$ thus $G^2= C \, \mathrm{rad}(G)$. When $\xi$ is chosen as above and $f\in A_+$,
we denote by $S_f(\xi)$ the following set of elements of $K$ (note that the elements $B$ and $G$ are uniquely determined
by $\xi$):
$$ S_f(\xi) = \set[\bigg]{ 
\frac{b+fG\xi}{a} \given  \begin{array}{l} a\in A_+,\,  b\in A; \exists  c\in A, ac= b^2+bfG+f^2\mathrm{rad}(G)B, \\
\vert b\vert<\vert a\vert \leq \vert c\vert,\, \gcd(a,b,f)=1 . \end{array} }. $$


\noindent Then, we have the following result, which is similar to Lemma~\ref{Ord}:
\begin{lem}\label{Ord2}${}$\par
Let $\xi\in\overline{k}$ be as above, and put $K=k(\xi)$. Let $f\in A_+$, and denote by
${\mathbb J}_{CM}(\xi,f)$ the set of singular moduli having complex multiplication by the order of $K$ of
conductor $f$. Let $z\in S_f(\xi)$. Then:
\begin{enumerate}
    \item\label{Ord2-1} $\vert\xi \vert\ge 1$ and $|z|=| z|_i=| z|_A\geq 1$.    
    \item\label{Ord2-2} ${\mathbb J}_{CM}(\xi,f)=\{ \sigma (j(z))\mid \sigma \in {\rm Gal}(K^{sep}/K)\}=j(S_f(\xi)).$
    \item\label{Ord2-3} If $\alpha_1,\ldots,\alpha_m$ denote the conjugates of $j(z)$ over $k$, we have:
    $${\mathbb J}_{CM}(\xi,f)=\{\alpha_1,\ldots,\alpha_m\}.$$
\end{enumerate}
\end{lem}
\begin{proof}This is a consequence of \cite {Gek83}, paragraph 4 (see also \cite{Sch97}, Theorem 3), \cite{HsYu98}, Lemma 5.1,  and direct computations. See also the proof of Theorem 7.5.21 of \cite{Pap}.
\end{proof}


Let us now consider the case of an inseparable quadratic extension $K/k$.
As already mentioned, there is actually a unique
such extension, namely, $K=\BF_q(\sqrt{T})$. In the sequel we denote by $F$ this extension.
Then $F/k$ is totally imaginary, $\infty=1/T$ is ramified in $F/k$, and we have:
$$\mathcal{O}_F=\mathbb F_q[\sqrt{T}]= A[\sqrt{T}].$$
For every $f\in A_+$ there is a unique order in $F$ of conductor $f$, namely, $\mathcal{O}= A[f\sqrt{T}]=A\oplus Af\sqrt{T}$.
Let $S_f(\sqrt{T})$ denote the following set of elements of $F$:
$$ S_f(\sqrt{T}) = \left\{ \frac{b+f\sqrt{T}}{a} \mid a\in A_+, b\in A; \exists c\in A, ac= b^2+f^2T,\,
\vert b\vert <\vert a \vert \leq \vert c\vert,\, \gcd(a,b,f)=1 \right\}. $$

We have:

\begin{lem}\label{Ordins2}${}$\par
Let $f\in A_+$, denote by
${\mathbb J}_{CM}(\sqrt{T},f)$ the set of singular moduli having complex multiplication by the order of $F$ of
conductor $f$, and let $z\in S_f(\sqrt{T})$,  Then:
\begin{enumerate}
    \item\label{Ordins2-1} $|z|=| z|_i=| z|_A\geq 1$.
    \item\label{Ordins2-2} We have:
$${\mathbb J}_{CM}(\sqrt{T},f)=\{ \sigma (j(z))\mid \sigma \in {\rm Gal}(F^{sep}/F)\}=j(S_f(\sqrt{T})).$$
\end{enumerate}
\end{lem}
\begin{proof} This is a consequence of \cite{Sch97}, Theorem 3, \cite{HsYu98}, Lemma 5.1,  and direct computations.
\end{proof}

\begin{rem}
Here we note that when $z\in S_f(F)$, then $j(z)$ is inseparable and $[k(j(z)):k]=2 [F(j(z)):F]$, by \cite{Sch97}. We refer the reader to the appendix for a more careful study of the inseparability of $j(z)$.
\end{rem}

\subsection{Estimates on CM \texorpdfstring{$j$}{j}-invariants}

In this section we collect a few lemmas on the CM $j$-invariants. We use the same notations as in the previous sections.

\subsubsection{Case $q$ odd}

\begin{lem}\label{Vince} Let $D\in A$ be such that $K=k(\sqrt{D})$ is imaginary quadratic and let $\sqrt{D}\in \bar k$ be a fixed
square root of $D$. Let $z\in S_D$ such that $\vert z\vert =1$.
Then there is a unique element $e\in \mathbb F_{q^2}\setminus \mathbb F_q$ such that $\vert z-e\vert<1$.
Moreover, we have
\begin{enumerate}
    \item $e^2=\frac{\sgn(D)}{4}$;
    \item If $z\not=e$, then $\displaystyle | z-e |\geq \frac{1}{\sqrt{|D|}}.$
\end{enumerate}
\end{lem}
\begin{proof} 
 Write  $z=\frac{-b+\sqrt{D}}{2a}$ with $a\in A_+$, $(b,c)\in A^2$, $b^2-4ac=D$, $
\vert b \vert < \vert a\vert \leq \vert c\vert$, $\gcd(a,b,c)=1$. From $|z|=1$, we derive $ \sqrt{\vert D\vert}=\vert a\vert$, hence $\deg D$ is even. By the remark following the statement of Conditions (\ref{cond-imaginary}), this implies that the infinite place is inert in $K/k$. Hence $K_\infty= \mathbb F_{q^2}((\frac{1}{T})).$ Let $e\in \mathbb F_{q^2}$ be such that $\sgn(z)=e$. We have $\vert z-e\vert <1$ and since $|z|=1$, we get:
$$e^2=\frac{\sgn(D)}{4}.$$ This proves the first point. We deal with the second point now.
 Let us set  $\bar z = \frac{-b-\sqrt{D}}{2a}. $ Then $\sgn (\bar z)= -e,$ and, therefore  we get $\vert \bar{z}-e\vert = \vert \bar z\vert =1$. Since $z\not =e,$ we have $(z-e)(\bar z-e)\not =0$. Now we have, since $a\neq0$:
$$(z-e)(\bar z-e)= \frac{(b+2ea)^2-D}{4a^2}= \frac{4ac+4e^2a^2+4eab}{4a^2}= \frac{c+e^2a+eb}{a}.$$
Now observe that $c+e^2a+eb \in \mathbb F_{q^2}[T]\setminus\{0\}.$ This leads to $$\vert z-e\vert\geq\frac{1}{\vert a\vert}=\frac{1}{\sqrt{\vert D\vert}},$$ which is the claim.
\end{proof}

The following theorem of Brown will be crucial in the sequel:
\begin{theorem}\label{Bro} Let $z\in S_D$. Let $n\geq 0$ be the smallest integer such that $n\geq \log _q |z|$.
\begin{enumerate}
    \item If $\deg D$ is odd, then $n\geq 1$ and 
$$\vert j(z)\vert = q^{\frac{q+1}{2} q^n}.$$
\item If $\deg D$ is even, then:
\begin{enumerate}
\item if $n\geq 1$ we have $| j(z)|= q^{q^{n+1}}$; \par 
\item if $n=0$, then there exists a unique $e\in \mathbb F_{q^2}\setminus \mathbb F_q$ such that $|z-e| <1$, and we have
$|j(z)|=q^q | z-e|^{q+1}$.
\end{enumerate}
\end{enumerate}
\end{theorem}
\begin{proof} This is a direct consequence of Theorem 2.8.2 of \cite{Brown}.
\end{proof}

\subsubsection{Case $q$ even}

First we consider the separable case. We will need the following key lemma:
\begin{lem}\label{Vince2} Let $\xi\in\bar k$ be as in Section \ref{CMcar2}, let $f\in A_+$, and
let $z\in S_f(\xi)$ be such that there exists $e\in \mathbb F_{q^2}\setminus \mathbb F_q$ with $\vert z-e\vert< 1$. Then:
\begin{enumerate}
    \item\label{Vince2-1} $e=\sgn(\xi)$ and $e^2+e= \sgn(B)$;
    \item\label{Vince2-3} if $z\not =e$, then $| z-e| \geq \frac{1}{| fG|}.$
\end{enumerate}
\end{lem}
\begin{proof} We use the same notations as in Section~\ref{CMcar2}. Set $K=k(\xi)$ and write $z= \frac{b+fG\xi}{a}$,
where
$$a\in A_+,\ (b,c)\in A^2,\ ac= b^2+bfG+f^2\mathrm{rad}(G)B,\ \vert b\vert<\vert a\vert \leq \vert c\vert,\ \gcd(a,b,f)=1.$$
Since $\vert z-e\vert< 1$, we have $\vert z \vert=1$ and the image of $z$ in the residue field at $\infty$ is $e$. It follows that
$K_{\infty}=\BF_{q^2}((1/T))$ and thus $\infty$ is inert in $K/k$. We then deduce from (\ref{xi}) in Section~\ref{CMcar2}
that $\vert\xi \vert=1$. On the other hand we
 have $\vert z\vert=1=\vert fG\xi/a\vert$, hence $\vert a\vert = \vert fG\vert$.
We get $\sgn(z)=\sgn(\xi)$ and hence (\ref{Vince2-1}). \par
\noindent Let us now prove (\ref{Vince2-3}). Set $\bar{z}= \frac{b+fG(\xi+1)}{a}.$ We have:
$$| \bar z -e| =1.$$
 Observe that:
 $$z\bar z= \frac{c}{a}, z+\bar z= \frac{fG}{a}.$$
 Thus:
 $$(z-e)(\bar z-e)= \frac{c+\sgn(B)a+e(fG+a)}{a}.$$
 Therefore,  if $z\not =e,$ $c+\sgn(B)a+e(fG+a)\not =0,$ and we get:
 $$\vert z-e\vert =\vert (z-e)(\bar z-e)\vert\geq \frac{1}{\vert a\vert}=\frac{1}{\vert fG\vert}.$$
 The lemma follows.
\end{proof}
The following result is the analogue in characteristic 2 of Theorem \ref{Bro}:
\begin{theorem}\label{Bro2} Let $z\in S_f(\xi).$ Let $n\geq 0$ be the smallest integer such that $n\geq \log_q\vert z\vert.$
\begin{enumerate}
\item\label{Bro2-1} {\sl Case $\infty$ is ramified in $K/k$ ($\deg B>\deg C$).}  Then $n\geq 1$ and 
$$| j(z)|  = q^{\frac{q+1}{2} q^n}.$$
\item\label{Bro2-2} {\sl Case $\infty$ is inert in $K/k$ ($\deg B=\deg C$).} Then:\par
i) if $n\geq 1$ we have $| j(z)|= q^{q^{n+1}},$\par
ii) if $n=0,$ then there exists a unique $e\in \mathbb F_{q^2}\setminus \mathbb F_q$ such that $|z-e| <1$, and we have
$|j(z)|=q^q | z-e|^{q+1}$.
\end{enumerate}
\end{theorem} 
\begin{proof} See \cite{HsYu98}, Theorem 5.2.
\end{proof}

Now we move to the inseparable case.
\begin{prop}\label{Broins2} Let $z\in S_f(\sqrt{T}).$ Let $n\geq 0$ be the smallest integer such that $n\geq \log_q\vert z\vert$.
Then $n\ge 1$ and  
$$|j(z)| =q^{\frac{q+1}{2}q^n} .$$
\end{prop} 
\begin{proof} See \cite{HsYu98}, Proposition 5.3.
\end{proof}

\subsection{Places and heights}

To each (non trivial) place $v$ of $k$, we associate an absolute value $|\cdot|_v$ normalized as follows. 
If the place $v$ corresponds to
a monic irreducible polynomial $P\in A$, 
we define $|x|_v = q^{-\deg(P) v_P(x)}$ for any $x\in k$ (where
$v_P$ is the usual $P$-adic valuation
on $k$, i.e., $v_P(P)=1$).
There is one more place of $k$, denoted $\infty$, and for this place we set $|x|_\infty = q^{\deg(x)}$ for any $x\in k$.  

For a finite extension $K/k$ we denote by $M_K$ the set of places of $K$. A place $v\in M_K$ is called infinite if it is an extension of $\infty$, otherwise it is called finite. The set of finite and infinite places of $K$ are denoted by $M_K^f$ and $M_K^\infty$, respectively.

To each place $v\in M_K$ we associate its absolute value normalized so that, for every $x\in k$ we have $|x|_v = |x|_w$, where $w\in M_k$ lies beneath $v$.

To each place $v\in M_K$ we also associate the ramification index $e_v$ (so $|K|_v \subset q^{\frac{1}{e_v}\BZ}$), the residual degree $f_v$ and the local degree $n_v = [K_v:k_v] = e_v f_v$.

We have the following two important properties:
\begin{itemize}
	\item {Product Formula:} For every $x\in K-\{0\}$, $\displaystyle \prod_{v\in M_K} |x|_v^{n_v} = 1$.
	\item {Extension Formula:} For every $w\in M_k$ we have $\displaystyle{[K:k] = \sum_{v|w}n_v}$.
\end{itemize}


\bigskip

We will associate to a Drinfeld module the J-height, which is the object of study here, and the Taguchi height, which will help in proving our result on units. 


\subsubsection{Na\"ive heights}

Let $\Phi$ be a Drinfeld $A$-module of rank $2$ over $K$. It is characterised by 
\[
\Phi_T = T + g \tau + \Delta \tau^2 ,\qquad g, \Delta \in K, \; \Delta\neq 0.
\]
We refer to $g$ and $\Delta$ as the {\em coefficients} of $\Phi$. Let $j=\frac{g^{q+1}}{\Delta}$ be the $j$-invariant of $\Phi$, as defined in paragraph \ref{Drinfeld intro}.

For any $\alpha \in K$, we define the logarithmic Weil height of $\alpha$:
\begin{equation}
h(\alpha) = \frac{1}{[K:k]}\sum_{v\in M_K} n_v \log_q\max\{ 1, |\alpha|_v\}.
\end{equation}
We now define the logarithmic Weil height of $\Phi$ as:
\begin{equation}\label{j-height}
h(\Phi)= h(j).
\end{equation}
This height, and its local components $h^v(\Phi)$ for $v\in M_K$, are invariant under isomorphisms of~$\Phi$. They do not depend on the choice of the field $K$ containing  $j$, thanks to the extension formula.

\subsubsection{Taguchi height}

In \cite{Tag} Taguchi defines the {\em differential height} of $\Phi$ as the degree of the metrised conormal line bundle along the unit section associated to a minimal model of~$\Phi$, which is the analogue of the Faltings height for abelian varieties.
All we need here is the identity (5.9.1) of \cite{Tag}, valid for Drinfeld modules of rank $r$ with everywhere stable reduction, which we adopt as our definition:
\begin{eqnarray}\label{Tagu}
\hTag(\Phi) &= & \frac{1}{[K:k]}\left[ \sum_{v\in M_K^f} h_G^v(\Phi) - \sum_{v\in M_K^\infty} n_v \log_q D(\Lambda_v)^{1/r}\right]
\end{eqnarray}
where for a finite place $v\in M_K^f$, the local component $h_G^v(\Phi) = n_v\log_q\max_i |g_i|_v^{1/(q^i-1)}$ equals Taguchi's $v(\Phi)$,
see \cite[\S2]{Tag}, $\Lambda_v$ is the lattice corresponding to the minimal model of $\Phi$ at the place $v$, see \cite[\S5]{Tag}, and $D(\Lambda_v)$ is the covolume of the lattice defined in \cite{Tag} on page 305. This height (\ref{Tagu}) is referred to as the \emph{stable Taguchi height} or the \emph{stable differential height}. It does not depend on the field of definition of the Drinfeld module, as long as $\Phi$ has everywhere stable reduction over that field. This is the height used by Wei in \cite{Wei20}. 

\subsection{Class numbers}

In the proof of Corollary~\ref{maj-hjb} we will need an explicit upper bound for the class number of an order
of an imaginary quadratic field $K$ in terms of its discriminant.
Since we will only need the case where the place $\infty$ is inert in $K/k$ we do not consider the case where $\infty$ is ramified,
but an analogous upper bound could be obtained using similar arguments in this case.
The main result of this section is the following estimate, which is valid in every characteristic.

\begin{prop}\label{class}
Let $\mathcal O$ be an order of an imaginary quadratic extension $K$ of $k$. Assume that the place $\infty$ is inert in $K/k$ and that
$\vert D_{\mathcal O}\vert > 1 $. 
Then we have:
$$h(\mathcal O) \leq \frac{37}{2(q+1)} \sqrt{\vert D_{\mathcal O}\vert}(\log_q\vert D_{\mathcal O}\vert)^2.$$
\end{prop}

In order to prove this proposition we need to introduce some more notation.
We will denote by $\mathcal P$ the set of monic irreducible polynomials in $A$.
Let $K$ be an imaginary quadratic extension of $k$ such that $\infty$ is inert in $K/k$. 
The extension $K/k$ is Galois (since $\infty$ is inert this is also the case when $q$ is even, see Section \ref{complexmultiplication}). Denote by $\chi : \Gal(K/k)\rightarrow \BC^{\times}$ the non-trivial character of $\Gal(K/k)$. When $P\in \mathcal P$, we
define $\chi(P)$ as follows. If $P$ is unramified in $K/k$ then $\chi(P) =\chi(\sigma_P)$, where
$\sigma_P\in {\rm Gal}(K/k)$ is the Frobenius at $P$.
If $P$ is ramified, we set $\chi(P)=0$.

We have:

\begin{lem}\label{Classnumber}
Let $\mathcal O$ be an order of an imaginary quadratic extension $K$ of $k$, and let $f\in A_+$ be its conductor. We have:
$$h(\mathcal{O})= h(\mathcal{O}_K)
\frac{\vert f \vert }{[\mathcal{O}_K^{\times}: \mathcal{O}^\times]} \prod_{\substack{P\in \mathcal P \\ P\mid f}} \left(1-\frac{\chi(P)}{| P|}\right).$$
\end{lem}
\begin{proof}
This formula can be proved as in the case of number fields, for instance following the proof of Theorem 7.24 page 146 of \cite{Cox}.
See also \cite[Proposition 17.9]{Ros} when $q$ is odd and \cite[Proposition 2.3]{Che} when $q$ is even.
\end{proof}

\begin{lem}\label{class2}
Let $K$ be an imaginary quadratic extension of $k$ such that $\infty$ is inert in $K/k$. Assume that $\deg D_K\geq 1$.
Then we have:
$$h(\mathcal{O}_K)\leq \frac{2}{q+1} \vert D_K\vert^{1/2}\deg D_K.$$
\end{lem}

\begin{proof}
When $q$ is odd, this inequality follows for example from
Lemma 9 of \cite{PC}. It turns out that the same proof also works when $q$ is even.
We give here the proof for the reader's convenience. In what follows we do not assume that $q$ is odd.

Note that since $\deg D_K \geq 1$ we have $K\not=\BF_{q^2}(T)$ and the constant field of $K$ is $\BF_q$.
Let $\zeta_K(s)$ be the zeta function of $K$. It is well-known that it is a rational function in $q^{-s}$. More precisely, by \cite[Theorem 5.9]{Ros}, 
there exists a polynomial $L_K(t)\in \mathbb Z[t]$ of degree $2g_K$,  where $g_K$ is the genus of $K$, such that 
$$\zeta_K(s) =\left.\frac{L_K(t)}{(1-t)(1-qt)}\right|_{t=q^{-s}}.$$
Moreover, if $h_K$ denotes the class number of $K$, we have
$$L_K(1)=h_K,$$
and we also have the famous functional equation:
$$L_K(t)= q^{g_K}t^{2g_K} L_K\Big(\frac{1}{qt}\Big).$$
This implies:
\begin{equation*}
h_K = q^{g_K} L_K (q^{-1}). 
\end{equation*}
Now, if $a\in A_+$ is the product of $r$ primes, say $a= P_1\cdots P_r$, set $\chi(a) =\chi(P_1)\cdots \chi(P_r)$ and define $\Lambda(\chi,t)$ by
$$\Lambda(\chi,t)= \displaystyle{\sum_{a\in A_+} \chi(a) t^{\deg (a)} \in \mathbb C[[t]].}$$
We have, for $\Re(s)>1$,
$$
\Lambda(\chi,q^{-s})=\sum_{a\in A_+} \frac{\chi(a)}{\vert a\vert^s}=\prod_{P\in \mathcal P} \left(1-\frac{\chi(P)}{\vert P\vert^s}\right)^{-1},
$$
and usual arguments on Artin $L$-series (see for instance Propositions 14.9 and 17.7 in \cite{Ros}) give us, for $\Re (s)>1$:
$$
L_K(q^{-s})=L(\chi,s)=
(1+q^{-s})^{-1}\Lambda(\chi,q^{-s}),
$$
where $L(\chi,s)$ denotes the Artin $L$-series associated to $\chi$.
Hence we get
$$\Lambda(\chi,t)= (1+t) L_K(t).$$
It follows that 
$\Lambda(\chi,t) \in \mathbb Z[t]$ with $\deg_t \Lambda(\chi,t)=2g_K+1$ and
\begin{equation*}
h_K = \frac{1}{q+1}q^{g_K+1}\Lambda(\chi,q^{-1}).
\end{equation*}
 Now:
$$\Lambda(\chi, q^{-1})=\sum_{\substack{a\in A_+ \\ \deg a\leq 2g_K+1}}\frac{\chi(a)}{\vert a\vert}\leq
\sum_{0\le k\le 2g_K+1}q^{-k}\sum_{\substack{a\in A_+ \\ \deg a =k}}1 = 2g_K+2.$$
 Thus we obtain:
 $$h_K\leq \frac{1}{q+1}q^{g_K+1}(2g_K+2).$$
Now, since the constant field of $K$ is $\BF_q$, we have $2g_K+2=\deg D_K$ by the Riemann-Hurwitz theorem (see e.g. Theorem 7.16 in \cite{Ros}). We have furthermore $h({\mathcal O}_K)=2 h_K$ since $\infty$ is inert: This follows for instance
from \cite[Proposition 14.1(b)]{Ros} as in the proof of \cite[Proposition 14.7]{Ros}. 
Thus we obtain the lemma.
\end{proof}

\begin{lem}\label{Mertens} Let $f\in A_+$ with $f\not =1$. We have:
$$\prod_{\substack{P\in \mathcal P \\ P \mid f}}\left(1-\frac{1}{\lvert P\rvert}\right)^{-1} \le e^{1+\frac{2/3}{\sqrt{q}-1}+\frac{1}{q-1}} \deg f \le 37 \deg f.$$
\end{lem}

\begin{proof}
For $n\geq 1,$ let $a_n$ be the number of monic irreducible polynomials of degree~$n$ in $A$. 
By \cite[Chapter~2, p.~14]{Ros} and since $q^{n/3} \leq \frac{1}{2} q^{n/2}$ for any $n\geq 6$, we get for any such $n$,
$$a_n\leq \frac{q^n}{n} + \frac{q^{n/2}}{n} + q^{n/3} \leq
\frac{q^n}{n} + \frac{q^{n/2}}{6} + \frac{q^{n/2}}{2}=\frac{q^n}{n} + \frac{2}{3} q^{n/2}.$$
Moreover, formula (3) on page~13 of \cite{Ros} (Corollary of Proposition~2.1) can be used to see that $a_n \leq q^n / n$ for any $n\in \{ 1,2,3,4,5\}$. Finally we obtain for any $n\geq 1$
\begin{equation}\label{eq-compterosen} a_n \leq \frac{q^n}{n} + \frac{2}{3} q^{n/2}.\end{equation}
We also have for any $n\geq 1$
\begin{equation}\label{eq-compterosen2}a_n \leq q^n,\end{equation}
as can be proved from~\eqref{eq-compterosen} if $n\geq 2$ since $q^{n/2} \leq \frac{1}{2} q^n$, and from $a_1=q$ if $n=1$.

Now let $m=\deg f$ and $L = \ln \Big(\displaystyle{\prod_{P\mid f}\left(1-\frac{1}{| P |}\right)^{-1}}\Big)$. We have
 \begin{align*} L&= - \sum_{P \mid f} \ln \left( 1 - \frac{1}{|P|}\right)= \sum_{P\mid f} \sum_{n\geq 1} \frac{1}{n|P|^n} = \sum_{P\mid f} \frac{1}{|P|} + \sum_{P\mid f} \sum_{n \geq 2} \frac{1}{n|P|^n} \\
 & \leq \sum_{P \mid f} \frac{1}{|P|} + \sum_{P \mid f} \sum_{n\geq 2} \frac{1}{2 |P|^n} = \sum_{P \mid f} \frac{1}{|P|} + \sum_{P \mid f} \frac{1}{2|P| (|P|-1)}\\
 & \leq \sum_{P \mid f} \frac{1}{|P|} + \sum_{P \mid f} \frac{1}{|P|^2} \qquad (\text{since } 2(|P|-1) \geq |P|) \\ & \leq \sum_{\substack{P \in \mathcal{P}\\ \deg P\leq m}} \frac{1}{|P|} + \sum_{\substack{P \in \mathcal{P}\\ \deg P \leq m}} \frac{1}{|P|^2}\\
 & \leq \sum_{k=1}^m \frac{a_k}{q^k} + \sum_{k=1}^m \frac{a_k}{q^{2k}}.
 \end{align*}
 By \eqref{eq-compterosen} we have $\frac{a_k}{q^k} \leq \frac{1}{k} + \frac{2}{3} q^{-k/2}$ and by \eqref{eq-compterosen2} we have $\frac{a_k}{q^{2k}} \leq q^{-k}$. Therefore
\begin{equation*}
    L\leq \sum_{k=1}^m \frac{1}{k} + \frac{2}{3} \sum_{k\geq 1} q^{-k/2} + \sum_{k\geq 1} q^{-k}
    \leq \ln m +1 + \frac{2/3}{\sqrt{q}-1} + \frac{1}{q-1}.
\end{equation*}

In conclusion we obtain
$$ \prod_{\substack{P \in \mathcal{P}\\ P \mid f}} \left(1 - \frac{1}{|P|}\right)^{-1} =e^L \leq e^{1+ \frac{2/3}{\sqrt{q}-1} + \frac{1}{q-1}} \deg f.$$
The second bound is derived from $e^{1+ \frac{2/3}{\sqrt{2}-1}  +1} \leq 37$.
\end{proof}

\begin{rem}
It follows from \cite[Theorem 3]{Ros9} that for every $\epsilon>0$, if $\deg f\gg_{q,\varepsilon} 0$ we have
$$\prod_{\substack{P\in \mathcal P\\ P \mid f}}\left(1-\frac{1}{\vert P\vert}\right)^{-1} \leq  e^{\gamma}(1+\varepsilon) \deg f,$$
where $\gamma$ is Euler's constant. However, the constant involved in the condition $\deg f\gg_{q,\varepsilon} 0$ is not explicitly given in \cite[Theorem 3]{Ros9}.
\end{rem}

We are now in position to prove Proposition~\ref{class}.

\begin{proof}[Proof of Proposition~\ref{class}]
Let $f$ denote the conductor of $\mathcal O$. Consider first the case where $\deg D_K\ge 1$. If $f=1$, then ${\mathcal O}={\mathcal O}_K$ and the result follows from Lemma~\ref{class2}. Suppose now that $f\not=1$.
Lemmas \ref{Classnumber} and \ref{class2}, together with the fact that the character
$\chi$ takes its values in $\{-1,0,1\}$, yield
\begin{align*}
h(\mathcal{O}) & \leq \frac{2}{q+1} \vert D_K\vert^{1/2} (\deg D_K)\, \vert f \vert \prod_{\substack{P\in \mathcal P\\ P\mid f}} \left(1-\frac{\chi(P)}{| P|}\right)\\
 & \leq \frac{2}{q+1} \vert D_K\vert^{1/2} (\deg D_K)\, \vert f \vert \prod_{\substack{P\in \mathcal P\\ P\mid f}} \left(1+\frac{1}{| P|}\right)\\
 & \leq \frac{2}{q+1} \vert D_K\vert^{1/2} (\deg D_K)\, \vert f \vert \prod_{\substack{P\in \mathcal P\\ P\mid f}} \left(1-\frac{1}{|P|}\right)^{-1}.
\end{align*}
Using now Lemma \ref{Mertens}, we obtain :
$$
h(\mathcal{O})\leq \frac{37}{q+1} \vert D_K\vert^{1/2} (\deg D_K)\, \vert f \vert \deg (f^2).
$$
Since
$$ \deg(f^2)\,\deg D_K\leq \frac{1}{4}\Big(\deg(f^2)+\deg D_K\Big)^2$$
and since $D_{\mathcal O}=f^2D_K$, we get the proposition.

Suppose now that $\deg D_K=0$. Then $K=\BF_{q^2}(T)$ and $h(\mathcal{O}_K)=1$. Moreover, since
$\vert D_{\mathcal O}\vert >1$ we have ${\mathcal O}\not={\mathcal O}_K$, hence
${\mathcal O}^{\times}=\BF_q^{\times}$ and thus $[\mathcal{O}_K^*: \mathcal{O}^\times]=q+1$. Using now Lemma \ref{Classnumber}
and Lemma \ref{Mertens} as before, we obtain :
$$
h(\mathcal{O})\leq \frac{37}{q+1}\vert f\vert \deg f = \frac{37}{2(q+1)}\vert D_{\mathcal O}\vert^{1/2} \deg D_{\mathcal O}.
$$
This implies again the estimate of the proposition.
\end{proof}

\section{Effective Andr\'e-Oort cases}

The goal of this section is to prove Theorem~\ref{Effect1-oddchar-intro}.

\subsection{Equation \texorpdfstring{$XY=\gamma$}{XY=gamma}, the odd characteristic case}${}$\par

We recall the notation from the introduction $\mathbb J_{CM}= \{ j(\Phi)\mid $ $\Phi $ is a CM $A$-Drinfeld module of rank two$\}.$ We are now in position to prove Theorem \ref{Effect1-oddchar-intro} in the case of odd characteristics, which we recall here:
\begin{theorem}\label{Effect1-oddchar}  Let $q\geq3$ be odd. Let $\gamma\in \overline{\mathbb F}_q[T]\setminus \{0\}, $ $\deg \gamma\leq q^2-2.$ 
Then the  equation $j_1 \, j_2=\gamma$ has no solutions for $j_1,j_2\in \mathbb J_{CM}.$
\end{theorem}
\begin{proof} Let us assume that there exist two CM $A$-Drinfeld modules $\Phi_1,\Phi_2$ such that $$j(\Phi_1)\, j(\Phi_2)=\gamma.$$  By Lemma \ref{Ord} (2), there exist non square elements $D_1,D_2\in A$ such that the fields $k(\sqrt{D_1})$ and $k(\sqrt{D_2})$
are imaginary quadratic and
$$j(\Phi_i)=j(z_i) , z_i\in S_{D_i}, i=1,2.$$
We can assume that $\deg(D_1)\geq \deg(D_2)$. We also observe that $\deg D_2\ge 1$, otherwise we would have
$z_2\in \mathbb F_{q^2}\setminus \BF_q$ and $j(z_2)=0$, contradiction. We have now two cases.
\begin{enumerate}
    \item {\sl  Case $\deg  D_1$ is odd.} By (\ref{Ord-2}) of  Lemma \ref{Ord}, possibly replacing $\gamma$ by a polynomial of the same degree, we can assume that $z_2= \sqrt{D_2}$. Since $\deg D_2 \ge 1$,
    Theorem \ref{Bro} (1) and (2a) yield:
    $$ | j(z_1)j(z_2)| \geq q^{q(q+1)}.$$
 Since $|\gamma| \leq q^{q^2-2},$ we get a contradiction.
 \item {\sl  Case  $\deg  D_1$ is even.} Again by (\ref{Ord-2}) of Lemma \ref{Ord}, we can assume that $z_1=\sqrt{D_1},$  and we set $n=\frac{1}{2}\deg D_1\geq 1.$ By Theorem~\ref{Bro}:
 $$ | j(z_1) | = q^{q^{n+1}}.$$
 Since $q^2-2<q^{n+1},$ we have $\vert j(z_2)\vert<1$ and we
 get also by Theorem~\ref{Bro} that $\deg  D_2$ is even, $n\geq\frac{1}{2} \deg D_2\geq 1,$ and there exists $e\in \mathbb F_{q^2}\setminus \mathbb F_q$ such that $| z_2-e | <1.$ We have by Theorem~\ref{Bro}:
 $$ | j(z_2) | = q^q | z_2-e|^{q+1}.$$
Since by Lemma \ref{Vince} we have $\displaystyle \vert z_2-e\vert\ge \vert D_2\vert^{-1/2}\ge q^{-n}$, we get:
 $$\frac{1}{q^n} \leq  q^{\frac{1}{q+1}(q^2-2-q^{n+1}-q)}.$$
 Thus:
 $$q^{n+1}+q-q^2+2\leq n(q+1).$$
 This leads to a contradiction since $n\geq 1$ and $q\geq 3.$
\end{enumerate}
\end{proof}
\begin{rem}\label{Hayes8} Let us  recall an example due to D. Hayes (\cite{Hay79}, Example 11.2). Let $q=3$ and let $\Phi$ be the rank two Drinfeld module given by:
$$\Phi_T=T+\sqrt{T-T^2}\eta\tau + \bar \eta \tau^2,$$
where $\eta =1+T+\sqrt{T^2-T}, \bar \eta = 1+T-\sqrt{T^2-T}.$ Then $\Phi$ is a CM Drinfeld module and ${\rm End}_{\mathbb C_\infty}\Phi =A[\sqrt{T-T^2}].$ We have:
$$ j(\Phi)=(T-T^2)^{2} \eta ^{5} .$$
Furthermore  if $\bar \Phi_T= T+\sqrt{T-T^2}\bar \eta\tau+ \eta \tau^2,$ then  ${\rm End}_{\mathbb C_\infty}(\bar \Phi) =A[\sqrt{T-T^2}],$ and:
$$  j(\Phi) j(\bar \Phi)= (T-T^2)^{4},$$ which is a solution in CM $j$-invariants when the right hand side is the polynomial $\gamma=(T-T^2)^{4}$ of degree 8 in $T$.
Observe that  $8=3^2-1,$ thus the bound $q^2-2$ in the above theorem is optimal.
\end{rem}

\subsection{Equation \texorpdfstring{$XY=\gamma$}{XY=gamma}, the characteristic two case}\

We start this paragraph with a crucial lemma due to Schweizer in the CM case, that we generalise in Proposition \ref{separabilityj} of the appendix. 

\begin{lem}\label{separabilityj-Schweizer}
Let $z$ be corresponding to a CM Drinfeld module of rank 2. We have the following equivalence:
$$z\in k_\infty^{sep} \iff  j(z) \in k_\infty^{sep}.$$
\end{lem}

\begin{proof}
This is due to Schweizer (\cite{Sch97}, Lemma 4). See also our Proposition \ref{separabilityj} of the appendix.
\end{proof}

We are now in position to prove  Theorem \ref{Effect1-oddchar-intro} in the case of characteristic 2:
\begin{theorem}\label{Effect1-car2}  Let $\gamma\in \overline{\mathbb  F}_q[T]\setminus \{0\}$ be such that
$\deg \gamma\leq q^2-2$. 
Then the  equation $j_1\,j_2= \gamma$ has no solutions for $j_1,j_2\in \mathbb J_{CM}.$
\end{theorem}
\begin{proof}  Let us  assume that there exist two CM $A$-Drinfeld modules $\Phi_1,\Phi_2$ such that $$j(\Phi_1)\, j(\Phi_2) =\gamma.$$
By Lemmas \ref{Ord2} (\ref{Ord2-2}) and \ref{Ordins2} (\ref{Ordins2-2}), there exist elements $\xi_1,\xi_2\in\bar k$,
$f_1, f_2\in A_+$, and $z_1\in S_{f_1}(\xi_1)$, $z_2\in S_{f_2}(\xi_2)$ such that:
$$j(\Phi_i)=j(z_i), i=1,2.$$\par
Now let us observe what happens if $\xi_1=\sqrt{T}$ (inseparable case). Then $z_1\not \in k_\infty^{sep},$ thus, by Lemma \ref{separabilityj-Schweizer}, $j(z_1)\not \in k_\infty ^{sep}$ and therefore $j(z_2)\not \in k_\infty^{sep}.$ This implies $z_2\not \in k_\infty^{sep}$ and therefore $\xi_2=\sqrt{T}$. By Proposition \ref{Broins2}, we get:
$$\deg  \gamma\geq  q(q+1).$$
This is a contradiction, and therefore we can assume that the fields $K_1=k(\xi_1)$ and $K_2=k(\xi_2)$ are separable extensions
of $k$. With the notation of Section~\ref{CMcar2}, we also may assume that $\deg (f_1G_1 )\geq \deg (f_2G_2).$ We have two cases.

\begin{enumerate}
    \item  {\sl  Case  $\infty$ is ramified in $K_1/k$}.\par
 \noindent
 In that case, Theorem \ref{Bro2} (1) gives:
 $$ | j(z_1) | \geq q^{\frac{q(q+1)}{2}}.$$
 By (\ref{Ord2-2}) of  Lemma~\ref{Ord2}, we can assume that $z_2= f_2G_2\xi_2.$
Let us observe that $\vert z_2\vert >1$. Indeed, otherwise by Lemma~\ref{Ord2} (\ref{Ord2-1}) we would have
$\vert z_2\vert=1$, hence $f_2=G_2=1$ and $\vert\xi_2\vert =1$. Then $z_2=\xi_2$ would be integral over $A$ by (\ref{xi}) (there
would be no denominator in the right-hand side of (\ref{xi}) since $G_2=1$),
hence we would have $z_2\in \mathbb F_{q^2}\setminus \BF_q$ and then $j(z_2)=0$, which is impossible.
It follows that we have,
by Theorem \ref{Bro2} (\ref{Bro2-1}) and (\ref{Bro2-2}.i): $$ \vert j(z_2)\vert \geq q^{\frac{q(q+1)}{2}}. $$
Thus we get $\displaystyle |j(z_1)j(z_2)| \geq q^{q(q+1)}$, which is a contradiction.
\item  {\sl  Case  $\infty$ is inert in $K_1/k$.} \par

 \noindent
 In that case, we easily see from Equation (\ref{xi}) in Section~\ref{CMcar2} that $\vert \xi_1\vert=1$.
 By (\ref{Ord2-2}) of  Lemma \ref{Ord2}, we can assume that $z_1= f_1G_1\xi_1$. Then we have $\vert z_1\vert = \vert f_1 G_1\vert$. 
 Set $n=\log_q\vert z_1\vert = \deg f_1+\deg G_1$. As before, we observe that $\vert z_1\vert >1$ hence $n\ge 1$.
 Hence by Theorem \ref{Bro2} (2), we have:
 $$ | j(z_1) | = q^{q^{n+1}}.$$
 Since $q^2-2<q^{n+1}$, we derive $\vert j(z_2) \vert<1$ and we get again by Theorem~\ref{Bro2}
 that there exists $e\in \mathbb F_{q^2}\setminus \mathbb F_q$ such that $| z_2-e| <1$, and:
 $$| j(z_2)| = q^q | z_2-e|^{q+1}.$$
Now, Lemma \ref{Vince2} (2) yields
 $$| z_2-e|^{q+1} \geq q^{-n(q+1)}.$$
We thus obtain
$$q^{n+1}+q-n(q+1) \leq q^2-2.$$
 This leads to a contradiction since $n\geq 1$.
\end{enumerate}
\end{proof}

\section{Singular units}

The goal of this section is to prove Theorem \ref{mainth}. We deal first with the ramified case, and then move to the more involved inert case.

\subsection{Proof of Theorem \ref{mainth} in the ramified case}

We prove here Theorem \ref{mainth} when $\infty$ is ramified in $K/k$, where $K$ is the field of complex multiplication of the singular modulus involved.
In this case we actually get a much better result, namely, that no singular modulus can be a unit.
This first case could be called the ``easy case'' since it is an easy consequence of the (difficult) result of Brown quoted above
(Theorem \ref{Bro} (1)) and its analogues in even characteristic (Theorem \ref{Bro2} (1) and Proposition \ref{Broins2}).

We state the result in the following proposition:

\begin{prop}\label{main-ramified}
Let $\alpha$ be a singular modulus. Let $K$ be the field of complex multiplication of $\alpha$, and suppose
that the place $\infty$ is ramified in $K/k$. Then $\alpha$ is not an algebraic unit.
\end{prop}

\begin{proof}
Write $\alpha=j(z)$, where $z\in\Omega$ is a CM point.
Consider first the case $q$ odd. Let $D\in A$ denote the discriminant of the ring
of complex multiplication $\mathcal{O}_z\subset K$. By Lemma \ref{CM1}, the element $\alpha$ is integral over $A$.
By Lemma \ref{Ord} (2) and (3), all conjugates $\alpha_i$ ($1\le i\le m$)
of $\alpha$ are of the form $\alpha_i=j(z_i)$ for some $z_i\in S_D$. Now, the result of Brown (Theorem \ref{Bro} (1))
yields $\vert \alpha_i \vert\ge q^{q(q+1)/2}>1$ for all $i$. It follows that $\alpha$ cannot be a unit.

If now $q$ is even, the same argument works by using Lemmas \ref{Ord2} and \ref{Ordins2} instead of Lemma~\ref{Ord},
and Theorem \ref{Bro2} (1) and Proposition \ref{Broins2} instead of Theorem \ref{Bro} (1).
\end{proof}

The rest of this section is devoted to the proof of Theorem \ref{mainth} when $\infty$ is inert in $K/k$.

\subsection{Congruence lemmas}

In this section, we prove elementary counting lemmas involving congruences, that will be needed in the proof of Proposition \ref{maj-hj}. They are similar to those proved in \cite{BHK}, Section 2.1.

\

As usual, if $a$ is a non zero element in $A$ and $P\in A$ is an irreducible polynomial, we denote by $v_P(a)$ the exponent of $P$
in the prime decomposition of $a$.

\begin{lem}\label{cardaps}
Suppose that $q$ is odd.
Let $D$ be a non zero element in $A$, let $s\ge 1$ be an integer, and let $P\in A$ be an irreducible polynomial.
Set $\nu=v_P(D)$ and $\displaystyle A_{P,s}=\{b\in A\mid b^2\equiv D \pmod{P^s}\}$. Then
the set $\displaystyle A_{P,s}$ consists of at most $2$ residue classes modulo
$\displaystyle P^{s-\min\{\lfloor \nu/2\rfloor,\lfloor s/2\rfloor\}}$.


\end{lem}

\begin{proof}
Let us first consider the case $\nu=0$, \ie{} $P\nmid D$. Then, for $b\in A$, we have:
$$
b^2\equiv D \pmod{P^s}\Longleftrightarrow \overline{b}^2 = \overline{D}\mbox{ in } \bigl(A/P^sA\bigr)^{\times},
$$
where we have denoted by $\overline{x}$ the residue class of $x$ in $A/P^sA$. Now, if $D$ is a square in
$\bigl(A/P^sA\bigr)^{\times}$, we know (see e.g. \cite{Ros}, proof of Proposition 1.10) that there are exactly two square roots
of $D$ in $\bigl(A/P^sA\bigr)^{\times}$, hence the result if $\nu=0$.

Suppose now that $1\le \nu< s$. If $A_{P,s}\not=\emptyset$, there exists $b\in A$ such that
$b^2\equiv D \pmod{P^s}$. We have then $v_P(b^2)=v_P(D)$ since $\nu<s$, and thus $\nu$ is even.
Write $b=P^{\nu/2}b'$ with $P\nmid b'$ and  $D=P^{\nu}D'$ with $P\nmid D'$.
Then we have $b'^2\equiv D' \pmod{P^{s-\nu}}$. The elements $b'\in A$ satisfying such a congruence 
consist of 2 residue classes modulo $P^{s-\nu}$ by the case $\nu=0$. This means that
the set $\displaystyle A_{P,s}$ consists of $2$ residue classes modulo $P^{s-\nu+ \nu/2}$, hence the result.

Suppose finally that $\nu\ge s$. Then $D\equiv 0 \pmod{P^s}$, and
$$
A_{P,s}=\{b\in A\mid b^2\equiv 0 \pmod{P^s}\}=\{b\in A\mid b\equiv 0 \pmod{P^{\lceil s/2\rceil}}\}.
$$
Thus the set $\displaystyle A_{P,s}$ consists of the class of $0$ modulo $P^{\lceil s/2\rceil}=P^{s-\lfloor s/2\rfloor}$,
which implies again the lemma.
\end{proof}

In even characteristic, we will use the following lemma:

\begin{lem}\label{cardaps2}
Suppose that $q$ is even.
Let $\delta$ and $\mu$ be two elements in $A$ with $\delta\not=0 $, let $s\ge 1$ be an integer, and let $P\in A$ be an irreducible polynomial.
Set $\nu=v_P(\delta)$ and $\displaystyle A_{P,s}=\{b\in A\mid b^2 +\delta b \equiv \mu \pmod{P^s}\}$. Then
the set $\displaystyle A_{P,s}$ is contained in at most $2$ residue classes modulo
$\displaystyle P^{s-\min\{\nu,\lfloor s/2\rfloor\}}$.
\end{lem}

\begin{proof}
Assume that $\displaystyle A_{P,s}\not=\emptyset$, and fix $b_0\in \displaystyle A_{P,s}$. Let $b\in A$ be another
element of $\displaystyle A_{P,s}$. Then we have:
\begin{equation}\label{congruence-mod-P-s}
b^2-b_0^2+\delta (b-b_0)=(b-b_0)(b-b_0+\delta)\equiv 0 \pmod{P^s}.    
\end{equation}
Suppose first that $\nu\leq s/2$. If $v_P(b-b_0)\le \nu$ then (\ref{congruence-mod-P-s}) gives $P^{s-\nu}\mid b-b_0+\delta$. Hence
$$
b\equiv b_0-\delta \pmod{P^{s-\nu}}
$$
and $b$ belongs to the residue class of $b_0-\delta$ modulo $\displaystyle P^{s-\nu}$.
If $v_P(b-b_0) \ge \nu+1$, then $v_P(b-b_0+\delta)=\nu$ and (\ref{congruence-mod-P-s}) yields
$$
b\equiv b_0 \pmod{P^{s-\nu}}.
$$
Thus $b$ belongs to the residue class of $b_0$ modulo $\displaystyle P^{s-\nu}$ in that case.

Suppose now that $\nu > s/2$. If $s$ is even then clearly (\ref{congruence-mod-P-s}) implies
$$
P^{s/2}\mid b-b_0\qquad {\rm or}\qquad P^{s/2}\mid b-b_0+\delta.
$$
If $s$ is odd, write $s=2k+1$ (with $k\in\BN$). Then we derive from (\ref{congruence-mod-P-s}):
$$
P^{k+1}\mid b-b_0\qquad {\rm or}\qquad P^{k+1}\mid b-b_0+\delta.
$$
Thus in both cases ($s$ even or odd) we obtain
$$
b\equiv b_0 \pmod{P^{\lceil s/2\rceil}} \qquad {\rm or}\qquad b\equiv b_0-\delta  \pmod{P^{\lceil s/2\rceil}}.
$$
Noticing now that $\lceil s/2\rceil = s-\min\{\nu,\lfloor s/2\rfloor\}$ since $\nu > s/2$, we find that
$b$ belongs to the residue classes of $b_0$ or $b_0-\delta$ modulo $\displaystyle P^{s-\min\{\nu,\lfloor s/2\rfloor\}}$.
\end{proof}

\begin{rem}
As the proof shows, in Lemma~\ref{cardaps} the set $\displaystyle A_{P,s}$ (when not empty) consists of the union of one or two residue
classes. However, in Lemma~\ref{cardaps2}, the set $\displaystyle A_{P,s}$ is only \emph{contained} in some residue classes:
one can show that it does not consist of the union of residue classes in general.
\end{rem}

If $a$ and $b$ are two non zero elements of $A$, we denote in the sequel by $\gcd_2(a,b)$ the polynomial $d\in A_+$ of maximal degree such that $d^2\mid a$ and $d^2\mid b$.

\begin{lem}\label{bmoda}
Let us denote by $\omega(a)$ the number of irreducible polynomials $P\in A_+$ such that $P\mid a$.
\begin{enumerate}
 \item\label{bmoda-1}
Suppose that $q$ is odd. Let $a$ and $D$ be two non zero elements in $A$.
Then the set $\{b\in A\mid b^2\equiv D \pmod{a}\}$ consists of at most $2^{\omega(a)}$ residue classes modulo
$a/\gcd_2(a,D)$.
 \item\label{bmoda-2}
Suppose that $q$ is even. Let $a$, $\delta$ and $\mu$ be elements in $A$ with $\delta\not=0$ and $a\not=0$.
Then the set $\{b\in A\mid b^2+\delta b \equiv \mu \pmod{a}\}$ is contained in at most $2^{\omega(a)}$ residue classes modulo
$a/\gcd_2(a,\delta^2)$.
\end{enumerate}
\end{lem}

\begin{proof}
Let us first prove (1). The lemma is clear if $\deg a=0$ (then $\omega(a)=0$), so we assume $\deg a\ge 1$ in the sequel.
Write $a=\displaystyle\prod_{1\le i\le r}P_i^{s_i}$, where the $P_i$'s are pairwise distinct and irreducible. With the notation of Lemma~\ref{cardaps} we have:
$$
\{b\in A\mid b^2\equiv D \pmod{a}\}=\bigcap_{1\le i\le r} \{b\in A\mid b^2\equiv D \pmod{P_i^{s_i}}\}=\bigcap_{1\le i\le r}A_{P_i,s_i}.
$$
It is easy to see that
$\displaystyle \gcd\nolimits_2(P_i^{s_i},D)=P_i^{\min\{\lfloor \nu_i/2\rfloor,\lfloor s_i/2\rfloor\}}$, where
$\nu_i=v_{P_i}(D)$. Thus, by Lemma~\ref{cardaps}, for each $i$ the set $\displaystyle A_{P_i,s_i}$
consists of at most $2$ residue classes modulo $\displaystyle P^{s_i}/\gcd\nolimits_2(P_i^{s_i},D)$.
Now, the intersection of classes modulo $\displaystyle P^{s_1}/\gcd\nolimits_2(P_1^{s_1},D),\ldots,P^{s_r}/\gcd\nolimits_2(P_r^{s_r},D)$ is a class modulo
$$
\prod_{1\le i\le r}P^{s_i}/\gcd\nolimits_2(P_i^{s_i},D) = a/\gcd\nolimits_2(a,D).
$$
It follows that the set $\{b\in A\mid b^2\equiv D \pmod{a}\}$ consists of at most $2^r$ residue classes modulo $a/\gcd_2(a,D)$.

Now, to show that the assertion (2) holds, it suffices to note that if we set $\nu'=v_P(\delta^2)$ in Lemma~\ref{cardaps2},
then we get exactly the same conclusion as in Lemma \ref{cardaps}, namely:
the set $\displaystyle A_{P,s}$ is contained in at most $2$ residue classes modulo
$\displaystyle P^{s-\min\{\lfloor\nu'/2\rfloor,\lfloor s/2\rfloor\}}$. Hence the same proof as before works in this case. 
\end{proof}

We deduce from Lemma~\ref{bmoda} the following corollary that will be crucial in Section \ref{upperbound}:

\begin{cor}\label{b-mod-a}
The following upper bounds hold.
\begin{enumerate}
\item Suppose that $q$ is odd.
Let $a$ and $D$ be two non zero elements in $A$, and let $\varepsilon$ be a real number such that $0<\varepsilon\le 1$.
Then
$$
\Card\{b\in A\mid b^2\equiv D \pmod{a},\ \vert b\vert<\epsilon\vert a \vert\} \le 2^{\omega(a)} \max\{1, q\varepsilon\vert \gcd\nolimits_2(a,D)\vert\}.
$$
\item Suppose that $q$ is even. Let $a$, $\delta$ and $\mu$ be elements in $A$ with $\delta\not=0$ and $a\not=0$.
let $\beta$ be an element in $k_{\infty}$, and let $\varepsilon$ be a real number such that $0<\varepsilon\le 1$. Then
$$
\Card\{b\in A\mid b^2+\delta b\equiv \mu \pmod{a},\ \vert b+\beta\delta\vert<\epsilon\vert a \vert\} \le 2^{\omega(a)} \max\{1, q\varepsilon\vert \gcd\nolimits_2(a,\delta^2)\vert\}.
$$
\end{enumerate}
\end{cor}

The proof of Corollary~\ref{b-mod-a} rests on the following easy counting lemma:

\begin{lem}\label{easycounting}
Let $m\in A$ be a non zero element, let $b_0\in A$, and let $M>0$ be a real number.
Then
\begin{equation*}
    \Card \{b\in A\mid b\equiv b_0\pmod{m},\, \vert b\vert < M\} \le \max\left\{1,\frac{qM}{\vert m\vert}\right\}.
\end{equation*}
\end{lem}

\begin{proof}
Without loss of generality we may assume that $\vert b_0\vert < \vert m\vert$.
Then we have
\begin{equation}\label{boundedclass}
\{b\in A\mid b\equiv b_0\pmod{m},\, \vert b\vert < M\} =\{\ell\in A\mid \vert b_0+\ell m\vert < M\}.
\end{equation}
This set is empty if $\vert b_0\vert \ge M$, and if $\vert b_0\vert < M$ we have
$$
\Card\{\ell\in A\mid \vert b_0+\ell m\vert < M\} =\Card\{\ell\in A\mid \vert \ell m\vert < M\} \le\max\left\{1,\frac{qM}{\vert m\vert}\right\},
$$
since if $C\ge 1$ is a real number, the number of polynomials $\ell\in A$ such that $\vert \ell \vert < C$ is at most
$\displaystyle q^{\lfloor \log_q C\rfloor +1}\le qC$.
Thus in all cases, the number of elements of the set occurring in (\ref{boundedclass}) is at most $\max\{1,qM/\vert m\vert\}$.
\end{proof}

\

\begin{proof}[Proof of Corollary~\ref{b-mod-a}]
Part (1) follows immediately from Lemma~\ref{bmoda} (1) and Lemma~\ref{easycounting} applied with $m=a/\gcd_2(a,D)$ and $M=\varepsilon \vert a \vert$.
In order to prove Part (2), write $\beta\delta=b_1+\zeta$ with $b_1\in A$ and $\vert \zeta\vert <1$.
Then, for every $b\in A$, we have
\begin{equation*}
    \vert b+\beta\delta \vert = \vert b+b_1+\zeta\vert =
    \begin{cases}
    \vert b+b_1\vert & \text{if $b\not=-b_1$}\\
    \vert \zeta\vert & \text{if $b=-b_1$.}
    \end{cases}
\end{equation*}
It follows that if $\varepsilon \vert a \vert \le 1$, the set 
\begin{equation}\label{setc}
    \{b\in A\mid b^2+\delta b\equiv \mu \pmod{a},\ \vert b+\beta\delta\vert<\epsilon\vert a \vert\}
\end{equation}
consists of at most the element $b=-b_1$ and the bound of Corollary~\ref{b-mod-a} (2) holds in this case.
If $\varepsilon \vert a \vert > 1$, the cardinality of the set (\ref{setc}) is equal to
\begin{multline*}
\Card\{b\in A\mid b^2+\delta b\equiv \mu \pmod{a},\ \vert b+b_1\vert<\epsilon\vert a \vert\} \\
= \Card\{b'\in A\mid b'^2+\delta b'\equiv \mu' \pmod{a},\ \vert b'\vert<\epsilon\vert a \vert\},
\end{multline*}
where $\mu'=\mu+b_1^2+\delta b_1$. The result now follows from
Lemma~\ref{bmoda} (2) and Lemma~\ref{easycounting} applied with $m=a/\gcd_2(a,\delta^2)$ and $M=\varepsilon \vert a \vert$.
\end{proof}

\subsection{Analytic estimates}

We will need to bound $\omega(a)$ from above, and for this we will need the following lemma.

\begin{lem}\label{maj-omega}
Let $a\in A$, $a$ be non constant,  and $d(a)=\Card\{d\in A_+\mid d\vert a\}$. If $\deg a \geq 2$, we have
$$
\log_qd(a)\le 15 \frac{\log_q\vert a\vert}{\log_q\log_q \vert a\vert}.
$$
\end{lem}

\begin{proof}

We follow the proof of the analogous result in characteristic zero
given in \cite[Chap. I.5, Th\'eor\`eme 2, p. 84]{Ten}. In the sequel, the letter $P$ will stand for a monic irreducible polynomial in $A$.
We will use the classical notation $P^{\nu}\mid\mid a$ to mean that $P^{\nu}\mid a$ but $P^{\nu+1}\nmid a$.

Let $t\in \BR$ such that $1< t\le\vert a\vert$. Write $\displaystyle a=\lambda\prod_{P^{\nu}\mid\mid a}P^{\nu}$,
where $\lambda\in\BF_q^{\times}$ and where the product is taken over all pairs $(P,\nu)$ with $\nu\in\BN$ and $P$ a monic irreducible polynomial such that $P^{\nu}\mid\mid a$. Then, using the rough estimate $\nu+1\le 2^{\nu}$,
$$
d(a)=\prod_{P^{\nu}\mid\mid a}(\nu+1)\le \prod_{\substack{P^{\nu}\mid\mid a\\ \vert P\vert\le t}}(\nu+1)\, 
\prod_{\substack{P^{\nu}\mid\mid a\\ \vert P\vert> t}} 2^{\nu}
\le \prod_{\substack{P^{\nu}\mid\mid a\\ \vert P\vert\le t}}(1+\deg a)\,  \prod_{\substack{P^{\nu}\mid\mid a\\ \vert P\vert> t}} 2^{\nu}.
$$
Now,
\begin{align*}
\Card\{P\mid P^{\nu}\mid\mid a, \vert P\vert\le t\} & \le \Card\{b\in A_+\mid q\le \vert b\vert\le t\} \\
& = \sum_{1\le i\le  \log_q t}\Card \{b\in A_+ \mid \deg b=i\} \\
& = \sum_{1\le i \le \log_q t}q^i \le \frac{q}{q-1}t \le 2t.
\end{align*}
On the other hand, for $\vert P\vert>t$ and $\nu\ge 1$ we have $\displaystyle 2^{\nu}\le (\vert P\vert^{\nu})^{\ln 2/\ln t}$. Hence
$$
d(a) \le \bigl(1+\deg a\bigr)^{2t}\; \big\vert\prod_{P^{\nu}\mid\mid a}P^{\nu}\big\vert^{\ln 2/\ln t} = \bigl(1+\log_q\vert a\vert \bigr)^{2t} \vert a \vert ^{\ln 2/\ln t}.\\
$$
Using the inequality $\log_q(1+x)\le \displaystyle\frac{\ln 3}{\ln 2}\log_q x$ for $x\ge 2$, we deduce from the above inequality that we have, if $\vert a\vert \ge q^2$:
\begin{equation}\label{choose-t}
\log_q d(a) \le  \frac{2\ln 3}{\ln 2}t \log_q\log_q\vert a\vert+ \frac{\ln 2}{\ln t}\log_q\vert a\vert.
\end{equation}

Assume first that $\vert a \vert > q^{q^5}$. We choose $\displaystyle t=\frac{\log_q\vert a\vert}{(\log_q\log_q\vert a\vert)^2}$ in equation (\ref{choose-t}). It is easy to see that
\begin{equation}\label{ineg}
x - 2 \log_q(x) \geq x/3 \quad \mbox{\rm for } x\ge 5,
\end{equation}
so $\log_q t\ge x/3$ by (\ref{ineg}) with $x=\log_q\log_q\vert a\vert$. Hence we obtain:
\begin{align*}
\log_q d(a)  & \le \frac{\log_q\vert a\vert}{\log_q\log_q\vert a\vert}\Bigl(\frac{2 \ln 3}{\ln 2}+\frac{\ln 2}{\ln t}\log_q\log_q\vert a\vert\Bigr) \\
& \le   \frac{\log_q\vert a\vert}{\log_q\log_q\vert a\vert}\left(\frac{2 \ln3}{\ln 2}+ 3 \log_q 2\right)
\le 7\, \frac{\log_q\vert a\vert}{\log_q\log_q\vert a\vert}.
\end{align*}

Suppose now that $\vert a \vert\le q^{q^5}$. We choose $t=2$ in equation $(\ref{choose-t})$ and we get $$\log_q d(a) \le \frac{4\ln 3}{\ln 2} \log_q\log_q\vert a\vert+ \log_q\vert a\vert \le \frac{\log_q\vert a\vert}{\log_q\log_q\vert a\vert}\left(\frac{4\ln 3}{\ln 2}\frac{(\log_q\log_q\vert a\vert)^2}{\log_q\vert a\vert}+\log_q\log_q\vert a\vert\right).$$ Since $(\log_q x)^2 \leq 3x/2$ for any $x \geq 2$, this gives $$\log_q d(a) \le \left(\frac{6\ln 3}{\ln 2}+5\right)\frac{\log_q\vert a\vert}{\log_q\log_q\vert a\vert} \le 15 \frac{\log_q\vert a\vert}{\log_q\log_q\vert a\vert} . $$

This proves the lemma.
\end{proof}

\begin{cor}\label{majomega}
Let $a\in A$. If $\deg a \geq 2$, we have
$$
\omega(a)\le \frac{15}{\log_q 2}\, \frac{\log_q\vert a\vert}{\log_q\log_q \vert a\vert}.
$$
\end{cor}

\begin{proof}
Write $\displaystyle a=\lambda\prod_{P^{\nu}\mid\mid a}P^{\nu}$ where $\lambda\in\BF_q^{\times}$. The number of factors in the product is $\omega(a)$, hence
$$
2^{\omega(a)}\le \prod_{P^{\nu}\mid\mid a}(\nu+1)=d(a).
$$
Applying Lemma~\ref{maj-omega} yields the result. \end{proof}

\begin{lem}\label{maj-sigma1}
Let $f\in A$, $f\not=0$. Let $\sigma_1(f):=\displaystyle\sum_{\substack{d\in A_+\\ d\mid f}}\vert d\vert$. Then
$$
\frac{\sigma_1(f)}{\vert f\vert}\le \log_q\vert f\vert +1.
$$
\end{lem}

\begin{proof}
We have $\sigma_1(f)=\displaystyle\sum_{\substack{d\in A_+\\d\mid f}}\left\vert \frac{f}{d}\right\vert $, so
$$
\frac{\sigma_1(f)}{\vert f\vert}=\sum_{\substack{d\in A_+\\ d\mid f}}\frac{1}{\vert d\vert}\le \sum_{\substack{d\in A_+\\ \deg d\le \deg f}}\frac{1}{\vert d\vert}
=\sum_{0\le \delta\le\deg f}\sum_{\substack{d\in A_+\\ \deg d=\delta}}\frac{1}{q^{\delta}}=\deg f +1.
$$
\end{proof}

\begin{rem}
It is likely that we have $\frac{\sigma_1(f)}{\vert f\vert}\ll \log_q\log_q\vert f\vert$ as in characteristic zero (see \cite[p. 88]{Ten}). However the weaker upper bound from Lemma~\ref{maj-sigma1} is enough for our purpose. 
\end{rem}


\subsection{Upper bound for \texorpdfstring{$h(\alpha)$}{h(alpha)}}\label{upperbound}

The goal of this section is to prove the following proposition and its corollary:
 
\begin{prop}\label{maj-hj}
Let $\alpha$ be a singular modulus such that the place $\infty$ is inert in $K/k$, where
$K$ is the field of complex multiplication of $\alpha$.
Let $\mathcal O$ be the ring of complex multiplication of $\alpha$ and
denote by $D$ its discriminant. 
Suppose that $\vert D\vert\ge q^{4}$. If $\alpha$ is a unit, then
we have, for all real number $\varepsilon$ such that $0<\varepsilon \le 1$:
\begin{equation*}
h(\alpha)\le (q+1)\log_q(\varepsilon^{-1})+\frac{3q(q+1)\varepsilon}{4}\; \frac{\vert D\vert^{1/2}}{h({\mathcal O})}\;
\vert D\vert^{15/(2\log_q\log_q\vert D\vert^{1/2})}\bigl(\log_q\vert D\vert\bigr)^2.
\end{equation*}
\end{prop}

\


\begin{cor}\label{maj-hjb}
Let $\alpha$ be a singular modulus such that the place $\infty$ is inert in $K/k$, where
$K$ is the field of complex multiplication of $\alpha$.
Let $\mathcal O$ be the ring of complex multiplication of $\alpha$ and
denote by $D$ its discriminant.
Suppose that $\vert D\vert\ge q^4$. If $\alpha$ is a unit, then we have
\begin{equation*}
h(\alpha)\le (q+1) \log_q\Bigl(\frac{\vert D\vert^{1/2}}{h({\mathcal O})}\Bigr) + 10(q+1)\frac{\log_q\vert D\vert}{\log_q\log_q\vert D\vert^{1/2}}.
\end{equation*}
\end{cor}

\begin{proof} We apply Proposition \ref{maj-hj} with
$$
\varepsilon = q^{-1} \frac{h({\mathcal O})}{\vert D\vert^{1/2}}\;
\vert D\vert^{-15/(2\log_q\log_q\vert D\vert^{1/2})}\bigl(\log_q\vert D\vert\bigr)^{-2}.
$$
By Proposition~\ref{class}, we see that $\displaystyle \varepsilon\leq \frac{37}{2q(q+1)}  \vert D\vert^{-15/(2\log_q\log_q\vert D\vert^{1/2})}$, hence $0<\varepsilon\le 1$ since $\vert D\vert\ge q^4$.
Now, the choice of $\varepsilon$ gives:
\begin{equation*}
    \log_q(\varepsilon^{-1}) = 1 + \log_q\Bigl(\frac{\vert D\vert^{1/2}}{h({\mathcal O})}\Bigr) + \frac{15}{2}\frac{\log_q\vert D\vert}{\log_q\log_q\vert D\vert^{1/2}}
+2\log_q\log_q\vert D\vert.
\end{equation*}
We thus derive the following estimate: 
\begin{equation}\label{hhh}
h(\alpha)\le (q+1) \log_q\Bigl(\frac{\vert D\vert^{1/2}}{h({\mathcal O})}\Bigr)
+ \frac{15}{2}(q+1) \frac{\log_q\vert D\vert}{\log_q\log_q\vert D\vert^{1/2}}
+\frac{7}{4}(q+1) + 2(q+1)\log_q\log_q\vert D\vert.
\end{equation}
We observe now that for $x\ge 2$ we have
\begin{equation*}
    \frac{7}{4}\leq \frac{x}{\log_2x}\leq\frac{x}{\log_qx}
\end{equation*}
and
\begin{equation*}
    2\log_q(2x)\leq 2\log_2(2x)\leq \frac{4x}{\log_2x}\leq \frac{4x}{\log_qx}.
\end{equation*}
Applying these inequalities with $x=\log_q\vert D\vert^{1/2}$, Equation (\ref{hhh}) yields the result.
\end{proof}

The rest of this section is devoted to the proof of Proposition \ref{maj-hj}. We fix a singular modulus $\alpha\in\BC_{\infty}$,
we denote by $K$ its field of complex multiplication and we assume that $\infty$ is inert in $K/k$. We further denote by
$\mathcal O$ the ring of complex multiplication of $\alpha$, by $D$ its discriminant and by $f$ its conductor.
When $q$ is even, $K/k$ is separable and we write $K=k(\xi)$ with $\xi$ satisfying the conditions of the beginning of
Section~\ref{CMcar2}. We will also use the notations $B$, $C$, $G$ and $S_f(\xi)$ of this section.
Finally, we fix a real number $\varepsilon$ such that $0<\varepsilon \le 1$.

Recall that by Lemmas \ref{Ord} (2) and \ref{Ord2} (2), the element $\alpha$ can be written
$\alpha=j(z)$, where $z\in S_D$ if $q$ is odd and $z\in S_f(\xi)$ if $q$ is even.
If $q$ is odd, we introduce the following set:
$$
C_{\varepsilon}(D)= \{z\in S_{D}\mid \exists\, e\in\BF_{q^2}\setminus\BF_q, \vert z-e\vert<\varepsilon\}.
$$
If $q$ is even, we introduce the following set:
$$
C_{\varepsilon}(f,\xi)= \{z\in S_f(\xi)\mid \exists\, e\in\BF_{q^2}\setminus\BF_q, \vert z-e\vert<\varepsilon\}.
$$

We have the following lemma, which is in the same vein as Lemma 2.7 of \cite{BHK}:

\begin{lem}\label{majb}
Assume that $q$ is odd. Let $(a,b,c)\in A^3$ such that
$$a\in A_+,\ \  b^2-4ac=D \ \ \text{and}\ \ \vert b \vert < \vert a\vert\leq \vert c\vert.$$
Put $\displaystyle z=(-b+\sqrt{D})/2a$, and suppose that there exists
$e\in\BF_q^2\setminus\BF_q$ such that $\vert z-e\vert<\varepsilon$. Then we have:
\begin{enumerate}
\item\label{majb-1} $\vert a\vert = \vert D\vert^{1/2}$.
\item\label{majb-2} $\displaystyle\Big\vert\frac{\sqrt{D}}{2e}-a\Big\vert<\varepsilon \vert a \vert$ and $\vert b\vert<\varepsilon \vert a \vert$.
\end{enumerate}
\end{lem}

\begin{proof} Let us first prove (\ref{majb-1}). Since $\vert b\vert<\vert a\vert\le \vert c\vert$, we have $\vert b\vert^2 < \vert ac\vert$, hence
$\vert D\vert=\vert b^2-4ac\vert=\vert ac\vert$. It follows that $\vert\sqrt{D}\vert>\vert b\vert$, hence $\vert -b+\sqrt{D}\vert = \vert\sqrt{D}\vert$ and
$\displaystyle\vert z\vert = \frac{\vert\sqrt{D}\vert}{\vert a\vert}$.
Now, since $\vert z-e\vert<\varepsilon$ we have $\vert z-e\vert<1$, hence $\vert z\vert=\vert e\vert=1$. So we obtain $\vert a\vert = \vert D\vert^{1/2}$.

Let us now prove (\ref{majb-2}). 
We first claim that the following property holds: if $x$ and $y$ are two elements of
$\BF_q((T^{-1}))$ and $\lambda\in\BF_{q^2}\setminus\BF_q$, then $\vert x+\lambda y\vert=\max\{\vert x\vert, \vert y\vert\}$.
Indeed, this is clear if $x=y=0$ or if $\vert x\vert\not=\vert y\vert$ by the ultrametric inequality. If now $\vert x\vert=\vert y\vert\not=0$, write
$\displaystyle x=\sum_{i\ge n}x_iT^{-i}$ and $\displaystyle y=\sum_{i\ge n}y_iT^{-i}$, where $(x_i,y_i)\in\BF_q\times\BF_q$ and $-n=\deg x=\deg y$. We have
$x_n+\lambda y_n\not=0$ since $y_n\not=0$ and $\lambda\notin\BF_q$, hence $\vert x+\lambda y\vert=\vert(x_n+\lambda y_n)T^{-n}\vert=\vert T^{-n}\vert=
\vert x\vert = \vert y\vert$. So the claim is proved.

Now, let $m=\deg a = \deg\sqrt{D}$. We have $\sqrt{D} = \mu T^m (1+\delta_1)$ for some elements $\mu\in\BF_{q^2}\setminus\BF_q$
and $\delta_1 \in  \BF_q[[T^{-1}]]$ with $\vert \delta_1\vert<1$. It follows that
$z=-b/(2a)+\sqrt{D}/(2a)=-b/(2a) + \mu/2 (1 + \delta_2)$, where $\delta_2 \in  \BF_q[[T^{-1}]]$
satisfies $\vert \delta_2\vert<1$. Since $\vert z-e\vert <1$, we see that in fact $\mu/2=e$, and
$$
z-e=-\frac{b}{2a}+\left(\frac{\sqrt{D}}{2a}-e\right)\quad \mbox{with}\quad \frac{\sqrt{D}}{2a}-e=e\delta_2\in e\BF_q[[T^{-1}]].
$$
We deduce from this and the claim above that
\begin{equation}\label{ze}
\vert z-e \vert=\max\left\{\left\vert\frac{b}{2a}\right\vert,\left\vert\frac{\sqrt{D}}{2a}-e\right\vert \right\}<\varepsilon,
\end{equation}
from which we derive (\ref{majb-2}).
\end{proof}

In even characteristic, we have the following analogue of Lemma~\ref{majb}:

\begin{lem}\label{majb-even}
Assume that $q$ is even. Let $(a,b,c)\in A^3$ such that
$$a\in A_+,\ \  ac= b^2+bfG+f^2\mathrm{rad}(G)B\ \ \text{and}\ \ \vert b \vert < \vert a\vert \leq \vert c\vert.$$
Put $z=(b+fG\xi)/a$, and suppose that there exists
$e\in\BF_q^2\setminus\BF_q$ such that $\vert z-e\vert<\varepsilon$. Then we have:
\begin{enumerate}
\item $\vert a\vert = \vert fG\vert$.
\item $\xi-e\in k_{\infty}$. 
\item If we set $\beta=\xi-e$, then
$\vert fG-a\vert<\varepsilon \vert a \vert$ and $\vert b+\beta fG\vert<\varepsilon \vert a \vert$.
\end{enumerate}
\end{lem}

\begin{proof}
(1) Since we have assumed that $\infty$ is inert in $K/k$, the element $\xi$ satisfies an equation of the form
\begin{equation*}\label{xi2}
\xi^2+\xi =\frac{B}{C}
\end{equation*}
with $(B, C)\in A^2$ such that $\vert B/C\vert=1$. It easily follows that $\vert \xi \vert=1$. By the inequality $\vert z-e\vert <1$
we also have $\vert z\vert=1$. Since $\vert b/a\vert <1$, we thus obtain
$$
\left\vert\frac{fG}{a}\right\vert=\left\vert\frac{\xi fG}{a}\right\vert=\left\vert z-\frac{b}{a}\right\vert=\vert z \vert=1,
$$
hence we have Part (1).

(2) The element $\xi$ is of degree $2$ over $k_{\infty}$ and its minimal polynomial is
$X^2+X=\frac{B}{C}$, hence its conjugates over $k_{\infty}$ are $\xi$ and $\xi+1$. On the orther hand, the element $e$ is also of degree
$2$ over $k_{\infty}$, and by Lemma~\ref{Vince2} its minimal polynomial over $k_{\infty}$ is $X^2+X=\sgn (B)$, hence its conjugates are $e$ and $e+1$. It follows that if $\sigma$ is the non trivial element of
$\Gal (k_{\infty}(\xi)/k_{\infty})$ (note that $k_{\infty}(\xi)=\BF_{q^2}((1/T))$ since $\infty$ is inert in $K/k$), we have
$$
\sigma (\xi-e)=\sigma(\xi)-\sigma(e)=(\xi+1)-(e+1)=\xi-e.$$
Hence $\xi-e\in k_{\infty}$. 

(3) We have:
$$
a(z-e) = b+\beta fG + e(fG-a).
$$
Since $b+\beta fG\in k_{\infty}$ and $e(fG-a)\in e k_{\infty}$ with $e\notin \BF_q$, we deduce from this
$$
\vert a(z-e) \vert = \max\{\vert b+\beta fG \vert, \vert fG - a\vert \},
$$
hence the result.
\end{proof}

We are now in position to establish the following estimates:

\begin{lem}\label{cardC}
Suppose that $\vert D\vert\ge q^{4}$. Then:
\begin{enumerate}
\item If $q$ is odd, we have
$$
\Card C_{\varepsilon}(D) \le 3q\varepsilon \vert D\vert^{1/2}\vert D\vert^{15/(2\log_q\log_q\vert D\vert^{1/2})}\log_q\vert D\vert^{1/2}.
$$
\item If $q$ is even, we have
$$
\Card C_{\varepsilon}(f,\xi) \le 3q\varepsilon \vert D\vert^{1/2}\vert D\vert^{15/(2\log_q\log_q\vert D\vert^{1/2})}\log_q\vert D\vert^{1/2}.
$$
\end{enumerate}
\end{lem}

\begin{proof} (1) Suppose that $q$ is odd. By Lemma \ref{majb}(\ref{majb-1}) and (\ref{majb-2}), 
$\Card C_{\varepsilon}(D)$ is less or equal to the number of pairs $(a,b)\in A_+\times A$ such that
$$
\vert a\vert = \vert D\vert^{1/2},\quad b^2\equiv D\!\!\!\pmod{a},\quad \Big\vert\frac{\sqrt{D}}{2e}-a\Big\vert<\varepsilon \vert a \vert
\quad \mbox{and}\quad\vert b\vert<\varepsilon \vert a \vert.
$$
Then Corollaries \ref{b-mod-a} and \ref{majomega} yield
\begin{align}
\Card C_{\varepsilon}(D) & \le \max_{\substack{a\in A_+\\ \vert a\vert = \vert D\vert^{1/2}}}\!\!2^{\omega(a)}\,\,
\Bigl(\!\!\!\sum_{\substack{a\in A_+\\ \vert a\vert = \vert D\vert^{1/2}\\ \vert\frac{\sqrt{D}}{2e}-a\vert<\varepsilon \vert a \vert}}
\max\{1,q\varepsilon \vert \gcd\nolimits_2(a,D)\vert \}\Bigr) \nonumber \\ \label{mmaajj}
& \le \vert D\vert^{15/(2\log_q\log_q\vert D\vert^{1/2})}\, 
\Bigl(\!\!\!\sum_{\substack{a\in A_+\\ \vert a\vert = \vert D\vert^{1/2} \\ \vert\frac{\sqrt{D}}{2e}-a\vert<\varepsilon \vert a \vert}}\!\!\!\!\!\! 1 \
+ \ q\varepsilon\!\!\! \sum_{\substack{a\in A_+\\ \vert a\vert = \vert D\vert^{1/2}}}\vert \gcd\nolimits_2(a,D)\vert\Bigr).
\end{align}
Let us first estimate the first sum in the parentheses. Let $a\in A_+$ be such that $\vert a\vert = \vert D\vert^{1/2}$ and $\vert\frac{\sqrt{D}}{2e}-a\vert<\varepsilon \vert a \vert$.
Recall that we have seen in the proof of Lemma~\ref{majb}(\ref{majb-2}) that $\frac{\sqrt{D}}{2e}\in\BF_q[[T^{-1}]]$.
Let $m=\deg a=\deg\sqrt{D}$, and write $a=\displaystyle\sum_{0\le i\le m}a_iT^{m-i}$ and $\displaystyle\frac{\sqrt{D}}{2e}=\sum_{i\ge 0}\delta_iT^{m-i}$,
where $a_i\in\BF_q$, $\delta_i\in\BF_q$ and $a_0=1$. Since $\vert\frac{\sqrt{D}}{2e}-a\vert<\varepsilon \vert a \vert$, we have $a_i=\delta_i$ for all $i$ such that
$q^{m-i}\ge \varepsilon\vert a\vert$, \ie\ for all $i$ such that $0\le i\le \lfloor\log_q(\varepsilon^{-1})\rfloor$. So the number of elements $a\in A_+$ such that
$\vert a\vert = \vert D\vert^{1/2}$ and $\vert\frac{\sqrt{D}}{2e}-a\vert<\varepsilon \vert a \vert$ is at most
$$
\Card\{i\in\BZ_{\ge 0}\mid \lfloor\log_q(\varepsilon^{-1})\rfloor<i\le m \}=q^{m-\lfloor\log_q(\varepsilon^{-1})\rfloor}\le q^{m+\log_q(\varepsilon)+1}=q\varepsilon\vert a \vert,
$$
hence
\begin{equation}\label{firstsum}
\sum_{\substack{a\in A_+\\ \vert a\vert = \vert D\vert^{1/2}\\ \vert\frac{\sqrt{D}}{2e}-a\vert<\varepsilon \vert a \vert}}\!\!\!\!\!\! 1 \ \le\  q\varepsilon\vert D \vert^{1/2}.
\end{equation}
Let us now estimate the second sum. We have
$$
\sum_{\substack{a\in A_+\\ \vert a\vert = \vert D\vert^{1/2}}}\vert \gcd\nolimits_2(a,D)\vert\le
\sum_{\substack{a\in A_+\\ \vert a\vert = \vert D\vert^{1/2}}}\sum_{\substack{d\in A_+\\ d^2\mid D, d^2\mid a}}\vert d\vert
=\sum_{\substack{d\in A_+\\ d^2\mid D}}\vert d\vert \Card\{a\in A_+\mid \vert a\vert = \vert D\vert^{1/2}, d^2\mid a\},
$$
and for a fixed $d$ such that $d^2\mid D$ we have
$$
\Card\{a\in A_+\mid \vert a\vert = \vert D\vert^{1/2}, d^2\mid a\}=\Card\left\{b\in A_+\mid \vert b\vert = \left\vert\frac{D^{1/2}}{d^2}\right\vert\right\}
=\frac{\vert D\vert^{1/2}}{\vert d\vert^2}.
$$
Thus we obtain:
\begin{equation*}\label{}
\sum_{\substack{a\in A_+\\ \vert a\vert = \vert D\vert^{1/2}}}\vert \gcd\nolimits_2(a,D)\vert\le\vert D\vert^{1/2}
\sum_{\substack{d\in A_+\\ d^2\mid D}}\frac{1}{\vert d\vert}.
\end{equation*}
Write now $D=f^2 D_0$, where $f\in A$ and $D_0\in A$ is squarefree. We have $d^2\mid  D\Longleftrightarrow d\mid f$, so Lemma~\ref{maj-sigma1} gives:
$$
\sum_{\substack{d\in A_+\\ d^2\mid D}}\frac{1}{\vert d\vert}=\sum_{\substack{d\in A_+\\ d\mid f}}\frac{1}{\vert d\vert}
=\frac{1}{\vert f\vert}\sum_{\substack{d\in A_+\\ d\mid f}}\left\vert \frac{f}{d}\right\vert=\frac{1}{\vert f\vert}\sum_{\substack{d\in A_+\\ d\mid f}}\vert d\vert
= \frac{\sigma_1(f)}{\vert f\vert}\le \log_q\vert f\vert +1\le 2\log_q\vert D\vert^{1/2}.$$
It follows that
\begin{equation}\label{secondsum}
\sum_{\substack{a\in A_+\\ \vert a\vert = \vert D\vert^{1/2}}}\vert \gcd\nolimits_2(a,D)\vert\le 2 \vert D\vert^{1/2}\log_q\vert D\vert^{1/2}.
\end{equation}
The inequality (\ref{mmaajj}) together with the estimates  (\ref{firstsum}) and (\ref{secondsum}) yield
$$
\Card C_{\varepsilon}(D) \le q\varepsilon \vert D\vert^{1/2}\vert D\vert^{15/(2\log_q\log_q\vert D\vert^{1/2})}(1+2\log_q\vert D\vert^{1/2}),
$$
from which the lemma follows.

(2) Let us now suppose that $q$ is even. The proof is very similar to that of the previous case.
Recall that if $z\in S_f(\xi)$ is such that there exists $e\in\BF_{q^2}\setminus\BF_q$ with $\vert z-e\vert<\varepsilon$, then
$e$ does not depend on $z$ and is given by $e=\sgn(\xi)$ by Lemma~\ref{Vince2}. Thus, if we set $\beta=\xi-\sgn(\xi)$,
by Lemma \ref{majb-even} we have that
$\Card C_{\varepsilon}(f,\xi)$ is less or equal to the number of pairs $(a,b)\in A_+\times A$ such that
$$
\vert a\vert = \vert fG\vert,\quad b^2+bfG \equiv f^2 \mathrm{rad}(G) B\!\!\!\pmod{a},\quad \vert fG-a\vert<\varepsilon \vert a \vert
\quad \mbox{and}\quad\vert b+\beta fG\vert<\varepsilon \vert a \vert.
$$
Similarly as before, since $D=f^2G^2$ (see Section \ref{CMcar2}), Corollaries \ref{b-mod-a} and \ref{majomega} yield
\begin{equation}\label{mmaajj2}
\Card C_{\varepsilon}(f,\xi) 
\le \vert D\vert^{15/(2\log_q\log_q\vert D\vert^{1/2})}\, 
\Bigl(\!\!\!\sum_{\substack{a\in A_+\\ \vert a\vert = \vert D\vert^{1/2} \\ \vert fG-a\vert<\varepsilon \vert a \vert}}\!\!\!\!\!\! 1 \
+ \ q\varepsilon\!\!\! \sum_{\substack{a\in A_+\\ \vert a\vert = \vert D\vert^{1/2}}}\vert \gcd\nolimits_2(a,D)\vert\Bigr).
\end{equation}
The sums in the parentheses are estimated in the same way as before. We get :

\begin{equation}\label{firstsum2}
\sum_{\substack{a\in A_+\\ \vert a\vert = \vert D\vert^{1/2}\\ \vert fG-a\vert<\varepsilon \vert a \vert}}\!\!\!\!\!\! 1 \ \le\  q\varepsilon\vert D \vert^{1/2}\qquad {\rm and} \qquad
\sum_{\substack{a\in A_+\\ \vert a\vert = \vert D\vert^{1/2}}}\vert \gcd\nolimits_2(a,D)\vert\le 2 \vert D\vert^{1/2}\log_q\vert D\vert^{1/2}.
\end{equation}
The inequalities (\ref{mmaajj2}) and (\ref{firstsum2}) give the result.
\end{proof}



\begin{proof}[Proof of Proposition \ref{maj-hj}]
We follow the strategy used in \cite{BHK} to prove their Theorem 3.1. We handle the cases $q$ odd and $q$ even simultaneously. When $q$ is even, we fix an element $\xi\in\overline{k}$ as in
Section~\ref{CMcar2} and we denote by $f$ the conductor of the ring of complex multiplication corresponding to $\alpha$.
Put 
$m=[k(\alpha):k]$, and denote by $\alpha_1,\ldots,\alpha_m$ the conjugates
of $\alpha$ over $k$. Let $\varepsilon$ be a real number such that $0<\varepsilon\le 1$.
By Lemmas \ref{Ord} and \ref{Ord2}, for all $i$ we can write $\alpha_i=j(z_i)$, where $z_i\in S_D$ if $q$ is odd
and $z_i\in S_f(\xi)$ if $q$ is even.
Suppose now that $\alpha$ is a unit. Then $\alpha^{-1}$ is integral and $\vert \alpha^{-1}\vert_v\le 1$ for every finite place $v\in M_{k(\alpha)}$. Hence we have
\begin{align*}
h(\alpha)  =h(\alpha^{-1}) & =\frac{1}{[k(\alpha):k]}\sum_{v\in M_{k(\alpha)}}n_v \log_q\max\{1,\vert\alpha^{-1}\vert_v\} \\
& =\frac{1}{[k(\alpha):k]}\sum_{v\mid\infty}n_v \log_q\max\{1,\vert\alpha^{-1}\vert_v\}\\
& =\frac{1}{m}\sum_{1\le i\le m}\log_q\max\{1,\vert\alpha_i^{-1}\vert\}\\
& = \frac{1}{m}\sum_{1\le i\le m}\log_q\max\{1,\vert j(z_i)^{-1}\vert\} =  S_1+S_2,
\end{align*}
where (recall that $\mathcal{E} = \mathbb{F}_{q^2} \backslash \mathbb{F}_q$)
\begin{equation*}\label{S1}
S_1=\frac{1}{m}\sum_{\substack{1\le i\le m\\ \min_{e\in{\mathcal E}}\vert z_i-e\vert \ge \varepsilon}}\log_q\max\{1,\vert j(z_i)^{-1}\vert\}
\end{equation*}
and
\begin{equation*}\label{S2}
S_2=\frac{1}{m}\sum_{\substack{1\le i\le m\\ \min_{e\in{\mathcal E}}\vert z_i-e\vert < \varepsilon}}\log_q\max\{1,\vert j(z_i)^{-1}\vert\}.
\end{equation*}
Let us first estimate the sum $S_1$. Let $i$ such that $\displaystyle\min_{e\in{\mathcal E}}\vert z_i-e\vert \ge \varepsilon$.
Since we are in the inert case, by Theorems \ref{Bro} (2) and \ref{Bro2} (2) we have
$\vert j(z_i)\vert\ge q^{q^2}$ if $\vert z_i \vert>1$,
and if $\vert z_i\vert =1$ we have $\vert j(z_i)\vert\ge q^q\varepsilon^{q+1}\ge \varepsilon^{q+1}$.
Hence, in all cases we have:
$$
\log_q\max\{1,\vert j(z_i)^{-1}\vert\}\le (q+1)\log_q(\varepsilon^{-1}).
$$
Since the number of terms in $S_1$ is at most $m$, we get the estimate:
\begin{equation}\label{SS1}
S_1\le (q+1)\log_q(\varepsilon^{-1}).
\end{equation}
Let us now estimate the sum $S_2$. Let $i$ be an index such that there exists $e\in{\mathcal E}$ with $\vert z_i-e\vert < \varepsilon$.
Then clearly $\vert z_i\vert=1$ since $\vert\varepsilon\vert<1$, and Theorems \ref{Bro} (2.b) and \ref{Bro2} (2.ii)
yield $\vert j(z_i)\vert = q^q\vert z_i-e\vert^{q+1}$. Using now Lemmas \ref{Vince} (2) and \ref{Vince2} (2), we get
$\vert j(z_i)\vert\ge q^q\vert D\vert^{-(q+1)/2}$. Hence we derive:
$$
\log_q\max\{1,\vert j(z_i)^{-1}\vert\}\le \frac{q+1}{2}\log_q\vert D\vert.
$$
Using the fact that $m=h({\mathcal O}_{D})$ by Lemma \ref{CM1}, we find the following estimate for $S_2$ if $q$ is odd:
\begin{equation}\label{SS2}
S_2\le \frac{q+1}{2} \, \frac{\Card C_{\varepsilon}(D)}{h({\mathcal O}_{D})} \log_q\vert D\vert
\end{equation}
and similarly if $q$ is even by replacing $C_\varepsilon(D)$ with $C_\varepsilon(f,\xi)$. Using now Lemma \ref{cardC}, the estimates (\ref{SS1}) and (\ref{SS2}) give the proposition.
\end{proof}

\subsection{Lower bound for \texorpdfstring{$h(\alpha)$}{h(alpha)}}\label{lowerbound}

In this section, we obtain two different lower bounds for the height of a singular modulus, given by Proposition~\ref{easybound} and Corollary~\ref{weibound}.

The first lower bound is easy:

\begin{prop}\label{easybound}
Let $D\in A$ be such that $K=k(\sqrt{D})$ is imaginary quadratic, and let $\alpha\in\BC_{\infty}$ be a singular modulus of discriminant $D$. If $\vert D\vert \ge q$, then we have:
$$
h(\alpha)\ge \frac{\vert D\vert^{1/2}}{h({\mathcal O}_{D})}.
$$
\end{prop}

\begin{proof} We start by $$h(\alpha)=\frac{1}{[K:k]}\sum_{v\in M_{K}}n_v \log_q\max\{1,\vert\alpha\vert_v\}.$$
So, when $q$ is odd we get: 
$$h(\alpha)\geq \frac{1}{h({\mathcal O}_{D})}\log_q \max_{z \in S_D}\{\vert j(z)\vert\} \geq
\frac{1}{h({\mathcal O}_{D})}  \log_q \vert j( \sqrt{D} / 2 )\vert \geq \frac{1}{h({\mathcal O}_{D})} \log_q q^{q^n}
$$
for $n$ the smallest integer such that $n\geq \log_q\vert \sqrt{D}/2\vert$, by Theorem \ref{Bro}.
When $q$ is even, we have similarly, using Theorem \ref{Bro2}:
$$h(\alpha)\geq
\frac{1}{h({\mathcal O}_{D})}  \log_q \vert j( \xi \sqrt{D})\vert \geq \frac{1}{h({\mathcal O}_{D})} \log_q q^{q^n}
$$
for $n$ the smallest integer such that $n\geq \log_q\vert \xi \sqrt{D}\vert \geq \log_q\vert \sqrt{D}\vert$ (recall
that $\vert \xi\vert \geq 1$ by Lemma \ref{Ord2}).
So, in both cases we get the bound of the proposition.
\end{proof}

\
The second lower bound will follow from the following proposition:

\begin{prop}\label{weiboundprop}
Let $K$ be a separable, imaginary quadratic extension of $k$, and let $\alpha\in\BC_{\infty}$ be a singular modulus 
having CM by an order of $\mathcal{O}_K$ of conductor $f$. Then we have: 

\begin{align*}
   h(\alpha)\geq &\quad  \frac{\log_q\vert D_K\vert}{10} \left(\frac{1}{2}-\frac{1}{\sqrt{q}+1}\right) \ln q  -\left[\frac{7q-5}{4q-4}+\frac{8}{\ln q}\right]\frac{\ln q}{5}\\
   &+\frac{1}{10}\log_q\vert f\vert -\frac{4q^2}{5(q-1)^2}\log_q^+ \log_q \vert f\vert,\\
\end{align*}
where we have set $\log_q^{+}(x)=\log_q\max\{1, x\}$ for $x\in\BR$.
\end{prop}

\begin{proof}
If $\alpha=0$ then $K=\BF_{q^2}(T)$, $\vert D_K\vert =1$ and $f=1$. Since $h(0)=0$, the lower bound of the proposition holds in that case.

From now on we assume that $\alpha\not=0$.
Let $\Phi$ be the Drinfeld module of rank 2 defined by
$$ \Phi_T=T + \tau + \frac{1}{\alpha}\tau^2.$$
The $j$-invariant of this Drinfeld module is equal to $\alpha$.
We have defined in (\ref{j-height}) the height of $\Phi$: $$h(\Phi) =h(\alpha), $$
and Definition 2.5 page 104 of \cite{DD} defines the height $h_{DD}(\Phi)$ by the formula $$h_{DD}(\Phi)=\max\{h(T), h(1), h(\alpha^{-1})\}=\max\{1, h(\alpha)\}.$$
We thus have 
\begin{equation*}\label{heightalphaphi}
    1+h(\alpha)\geq h_{DD}(\Phi)\geq h(\alpha).
\end{equation*}
Furthermore by \cite{DD} Lemme 2.14.(v) page 108, we have 
\begin{equation*}\label{DD}
5h_{DD}(\Phi) +1 \geq h_{\mathrm{Tag}}(\Phi),
\end{equation*}
hence
\begin{equation}\label{DDP}
5h(\alpha) \geq h_{\mathrm{Tag}}(\Phi)-6.
\end{equation}

Consider first the special case where $\Phi$ has CM by $\mathcal{O}_K$.
We then invoke the results of Fu-Tsun Wei. 
Apply Theorem~5.3 page 887 of \cite{Wei20}, Lemma 5.5 page 887 of \cite{Wei20}, and equation (5.5) page 889 of \cite{Wei20} to obtain 

\begin{equation}\label{Wei}
h_{\mathrm{Tag}}(\Phi)= \frac{g_K-1}{2}\ln q_K + (g_k-1)\ln q + \frac{1}{2} \ln q_{\infty} + \frac{\gamma_K}{2},
\end{equation}
where $g_K$ is the genus of $K$, $q_K$ is the cardinality of the constant field of $K$, $g_k$ is the genus of $k$,
$\gamma_K$ is the Euler-Kronecker constant of $K$ and $q_\infty=\# \mathbb{F}_\infty$.
We need an estimate on $\gamma_K$. We will use the work of Ihara \cite{Iha}, who gives in (1.4.6) page 420 the inequality 
\begin{equation}\label{gammaK}
\gamma_K\geq \frac{-2g_K}{\sqrt{q}+1}\ln{q}-c_q,
\end{equation}
where from (1.3.12) page 418 of \cite{Iha} we have 
\begin{equation}\label{cq}
    c_q=\frac{q+1}{2(q-1)}\ln{q}. 
\end{equation}
Since $g_k=0$ and $q_\infty = q$, combining (\ref{Wei}), (\ref{gammaK}), and $(\ref{cq})$ provides us with

\begin{equation}\label{Wei2}
   h_{\mathrm{Tag}}(\Phi)\geq \frac{g_K-1}{2}\ln q_K -\ln q + \frac{1}{2} \ln q -\frac{g_K}{\sqrt{q}+1}\ln{q} -\frac{q+1}{4(q-1)}\ln{q}.
\end{equation}
We distinguish now two cases.

Suppose first that $K\not=\BF_{q^2}(T)$. Then $q_K=q$ and (\ref{Wei2}) yields

\begin{equation*}
   h_{\mathrm{Tag}}(\Phi)\geq g_K\left(\frac{1}{2}-\frac{1}{\sqrt{q}+1}\right)\ln q  -\frac{5q-3}{4q-4}\ln{q}.
\end{equation*}

By the Riemann-Hurwitz theorem (Theorem 7.16 of \cite{Ros}) we have $g_K\geq\frac{1}{2}\log_q\vert D_K\vert -1$, hence we get:

\begin{equation}\label{lower_max_order_case1}
    h_{\mathrm{Tag}}(\Phi)\geq \frac{\log_q\vert D_K\vert}{2} \left(\frac{1}{2}-\frac{1}{\sqrt{q}+1}\right) \ln q  -\left(\frac{7q-5}{4q-4}-\frac{1}{\sqrt{q}+1}\right)\ln q.
\end{equation}

Suppose now that $K=\BF_{q^2}(T)$. Then $q_K=q^2$ and $g_K=0$, hence (\ref{Wei2}) gives
\begin{equation}\label{lower_max_order_case2}
   h_{\mathrm{Tag}}(\Phi)\geq - \frac{7q-5}{4q-4}\ln q.
\end{equation}
Since $\vert D_K\vert=1$ in that case, we see from (\ref{lower_max_order_case1}) and (\ref{lower_max_order_case2}) that we have, in both cases:
\begin{equation}\label{lower_max_order}
   h_{\mathrm{Tag}}(\Phi)\geq \frac{\log_q\vert D_K\vert}{2} \left(\frac{1}{2}-\frac{1}{\sqrt{q}+1}\right) \ln q  -\frac{7q-5}{4q-4} \ln q.
\end{equation}
By $(\ref{DDP})$, it follows that the bound of the proposition holds when $\Phi$ has CM by the maximal order $\mathcal{O}_K$.

In the general case, we use a strategy which is similar to the one of Habegger \cite{Hab15} in the case of CM elliptic curves. If $\Phi$ has CM by $\mathcal{O}=\mathbb{F}_q[T]+f \mathcal{O}_K$, order of conductor $f\in\mathbb{F}_q[T]$, then $\Phi$ is isogenous to another Drinfeld module $\Phi'$ with CM by the maximal order $\mathcal{O}_K$. We may thus reduce to the previous case using classical comparisons of heights for isogenous Drinfeld modules \cite{Tag}. In this particular case, we use Proposition \ref{Taguchi non-max} (see below, see also Ran \cite{ZR}), which gives:

\begin{equation}\label{Zhenlin}
 h_{\mathrm{Tag}}(\Phi)=h_{\mathrm{Tag}}(\Phi')+\frac{1}{2}\log_q\vert f\vert -\frac{1}{2}\sum_{v\vert f}\mathrm{deg}(v) e_{f}(v),
 \end{equation}
 where the sum is taken over $v$ prime and monic, and 
 \begin{equation*}
  e_{f}(v)=\frac{(1-\chi(v))(1-\vert v\vert^{-v(f)})}{(\vert v\vert -\chi(v))(1-\vert v\vert^{-1})}
 \end{equation*}
 where $\chi(v)=1$ if $v$ splits in $K$, $\chi(v)=0$ if $v$ ramifies in $K$, and $\chi(v)=-1$ if $v$ is inert in $K$. 
 
 We need an explicit upper bound on $e_{f}(v)$. We have, for any prime monic $v$, 
 $$
 \frac{(1-\chi(v))(1-\vert v\vert^{-v(f)})}{(\vert v\vert -\chi(v))(1-\vert v\vert^{-1})}\leq \frac{2}{\vert v\vert -1} \frac{q}{q-1}\leq \frac{2q^2}{(q-1)^2}\frac{1}{\vert v\vert},$$ which leads to 
 \begin{equation}\label{step}
 \frac{1}{2}\sum_{v\vert f}\mathrm{deg}(v) e_{f}(v)\leq \frac{q^2}{(q-1)^2} \sum_{v\vert f} \frac{\log_q\vert v\vert}{\vert v\vert}.
 \end{equation}

Combining (\ref{DDP}), (\ref{lower_max_order}), (\ref{Zhenlin}), and direct estimates on (\ref{step}) leads to the claim for general orders. Let us compute a valid explicit constant by bounding the right hand side of (\ref{step}) from above. We start with the particular case $\deg(f)=1$, where we get 
\begin{equation*}\label{degf1}
\frac{1}{2}\sum_{v\vert f}\mathrm{deg}(v) e_{f}(v)\leq \frac{q}{(q-1)^2}.
\end{equation*}
We will now assume $\deg(f)\geq2$ and continue with
$$\sum_{v\vert f} \frac{\log_q\vert v\vert}{\vert v\vert} \leq \sum_{\vert v\vert \leq \log_q\vert f\vert} \frac{\log_q\vert v\vert}{\vert v\vert}+\sum_{\vert v\vert > \log_q\vert f\vert} \frac{\log_q\vert v\vert}{\vert v\vert}.$$
For any $x\geq1$ we have
$$\sum_{v \leq x} \frac{\log_q\vert v\vert}{\vert v\vert} = \sum_{n=0}^{\lfloor \log_q x \rfloor} \frac{n a_n}{q^n},$$
where $a_n$ is the number of prime monic polynomials of degree $n$. A classical counting argument (see for instance \cite{Ros} page 14) gives for any $n\geq1$ 
$$\vert a_n \vert \leq \frac{q^n}{n} + \frac{q^{n/2}}{n}+q^{n/3}.$$
Hence 
$$\sum_{\vert v\vert \leq x} \frac{\log_q\vert v\vert}{\vert v\vert} \leq \sum_{n=1}^{\lfloor \log_q x \rfloor} \frac{n}{q^n} \Big(\frac{q^n}{n} + \frac{q^{n/2}}{n}+q^{n/3}\Big)\leq 3 \log_q x.$$
So we obtain 
\begin{equation}\label{vpetit}
\sum_{\vert v\vert \leq \log_q\vert f\vert} \frac{\log_q\vert v\vert}{\vert v\vert} \leq 3\log_q^+ \log_q \vert f\vert.
\end{equation}

As $x\mapsto \frac{\log_q x}{x}$ is a decreasing function on intervals where $x\geq3$ (in particular), we may continue, as we assumed $\deg(f)\geq2$, with
\begin{equation}\label{vgrand}
    \sum_{\substack{v\vert f \\ \vert v\vert > \log_q\vert f\vert}} \frac{\log_q\vert v\vert}{\vert v\vert} \leq \frac{\log_q^+\log_q\vert f\vert}{\log_q\vert f\vert} \sum_{\substack{v\vert f \\ \vert v\vert > \log_q\vert f\vert}} 1 \leq  \log_q^+\log_q \vert f\vert.
\end{equation}

Injecting (\ref{vpetit}) and (\ref{vgrand}) in (\ref{step}) gives 

\begin{equation*}
     \frac{1}{2}\sum_{v\vert f}\mathrm{deg}(v) e_{f}(v)\leq \frac{4q^2}{(q-1)^2}\log_q^+ \log_q \vert f\vert.
\end{equation*}

Going back to \eqref{Zhenlin} we get in all cases

\begin{equation*}\label{lower_general}
 h_{\mathrm{Tag}}(\Phi)\geq h_{\mathrm{Tag}}(\Phi')+\frac{1}{2}\log_q\vert f\vert -\frac{4q^2}{(q-1)^2}\log_q^+ \log_q \vert f\vert -\frac{q}{(q-1)^2},
 \end{equation*} 
and finally, using (\ref{DDP}) and (\ref{lower_max_order}):

\begin{align*}
   h(\alpha)\geq & \frac{\log_q\vert D_K\vert}{10} \left(\frac{1}{2}-\frac{1}{\sqrt{q}+1}\right) \ln q  -\left[\frac{7q-5}{4q-4}+\frac{6}{\ln q}\right]\frac{\ln q}{5}\\
   &+\frac{1}{10}\log_q\vert f\vert -\frac{4q^2}{5(q-1)^2}\log_q^+ \log_q \vert f\vert-\frac{2}{5},\\
\end{align*}
which gives the proposition. 
\end{proof}

We can now state our second lower bound for the height of a singular modulus.

\begin{cor}\label{weibound}
Let $K$ be a separable, imaginary quadratic extension of $k$, and let $\alpha\in\BC_{\infty}$ be a singular modulus
having CM by an order of $\mathcal{O}_K$ of discriminant $D$. Then:
$$h(\alpha)\geq \frac{\ln 2}{120}\log_q\vert D\vert -3\ln q-8.$$
\end{cor}

\begin{proof}
We start from Proposition \ref{weiboundprop}. We have, for any $q\geq2$,
\begin{equation}\label{logD}
    \frac{\log_q\vert D_K\vert}{10}\left(\frac{1}{2}-\frac{1}{\sqrt{q}+1}\right)\ln q \geq \frac{\log_q\vert D_K\vert}{120}\ln q
\end{equation}
and 
\begin{equation}\label{cst}
    -\left[\frac{7q-5}{4q-4}+\frac{8}{\ln q}\right]\frac{\ln q}{5}\geq -3\ln q.
\end{equation}
Moreover, when $\log_q\vert f\vert \geq 80$ we have
\begin{equation}
\frac{1}{10}\log_q\vert f\vert -\frac{4q^2}{5(q-1)^2}\log_q^+ \log_q \vert f\vert\geq \frac{\ln 2}{120}\log_q\vert f\vert^2,
\end{equation}
 and when $\log_q\vert f\vert \leq 80$ we can always bound from below 

\begin{equation}\label{logf}
\frac{1}{10}\log_q\vert f\vert -\frac{4q^2}{5(q-1)^2}\log_q^+ \log_q \vert f\vert\geq \frac{\ln 2}{120}\log_q\vert f\vert^2-8,
\end{equation}
which is thus valid in all cases, so overall by (\ref{logD}), (\ref{cst}), (\ref{logf}) we get the weaker

\begin{equation}
    h(\alpha)\geq \frac{\log_q\vert D_K\vert}{120}\ln q + \frac{\ln 2}{120}\log_q\vert f\vert^2-3\ln q -8,
\end{equation}
which implies the claim.

\end{proof}

\medskip

In the proof of Proposition~\ref{weiboundprop} we have used the following proposition. It is due to Ran and gives a formula that is the Drinfeld module analogue of the Nakkajima-Taguchi \cite{NT91} formula for CM elliptic curves. We recall it here for completeness.
\begin{prop}\label{Taguchi non-max}
Let $\Phi$ be a Drinfeld module of rank 2 that has CM by $\mathcal{O}=\mathbb{F}_q[T]+f \mathcal{O}_K$, order of conductor $f\in\mathbb{F}_q[T]$. Then $\Phi$ is isogenous to another rank 2 Drinfeld module $\Phi'$ with CM by the maximal order $\mathcal{O}_K$, and we have
\begin{equation}
 h_{\mathrm{Tag}}(\Phi)=h_{\mathrm{Tag}}(\Phi')+\frac{1}{2}\log_q\vert f\vert -\frac{1}{2}\sum_{v\vert f}\mathrm{deg}(v) e_{f}(v),
 \end{equation}
 where the sum is taken over $v$ prime and monic, and 
 \begin{equation}
  e_{f}(v)=\frac{(1-\chi(v))(1-\vert v\vert^{-v(f)})}{(\vert v\vert -\chi(v))(1-\vert v\vert^{-1})}
 \end{equation}
 where $\chi(v)=1$ if $v$ splits in $K$, $\chi(v)=0$ if $v$ ramifies in $K$, and $\chi(v)=-1$ if $v$ is inert in $K$. 
\end{prop}

\begin{proof}
We follow the proof of Nakkajima-Taguchi \cite{NT91}. Start by applying Lemma 5.5 page~309 of Taguchi's \cite{Tag}, which gives for the stable height of Drinfeld modules of rank 2 linked by an isogeny $u:\Phi'\longrightarrow\Phi$ of degree $f$

\[
h(\Phi)=h(\Phi')+\frac{1}{2}\log_q\vert f \vert-\frac{1}{[K:k]}\log_q\#(\mathcal{O}_F/D_u),
\]

\noindent where $D_u$ is the \textit{different ideal} of the isogeny $u$. The different ideal is in particular keeping track of the places $v$ where the isogeny is not locally \'etale: $\mathcal{O}_F/D_u$ has precisely these places as support (see \cite{Ray} page 203 and \cite{Tag} page 308-309 for the specific case of Drinfeld modules). It is computed in \cite{NT91} (1.5) page 122 by studying independently the ordinary primes and the supersingular primes, we have $$\#(\mathcal{O}_F/D_u)=\prod_{v\vert f} \vert v\vert^{e_f(v)[K:k]/2}.$$ This gives the proposition.
\end{proof}

\subsection{Proof of Theorem~\ref{mainth}}

In this section, we use the bounds obtained in Sections \ref{upperbound} and \ref{lowerbound} to finally prove Theorem~\ref{mainth}.

\noindent
\begin{proof}[Proof of Theorem~\ref{mainth}]
Let $\alpha$ be a singular modulus. Let $K$ be the corresponding field of complex multiplication, and denote by $\mathcal O$
the ring of complex multiplication of $\alpha$.
Suppose that $\alpha$ is an algebraic unit in $\overline{k}$.
Then the place $\infty$ is not ramified in $K/k$ by Proposition~\ref{main-ramified}, hence it is inert.
Let $D$ be the discriminant of the order $\mathcal O$ and $f$ be its conductor.
By Corollary~\ref{maj-hjb}, Proposition \ref{easybound} and Corollary \ref{weibound}, we have, when $\vert D\vert \geq q^{4}$:
\begin{equation}\label{laclef}
\max\left\{\frac{\vert D\vert^{1/2}}{h({\mathcal O})},\frac{\ln2}{120}\log_q\vert D\vert-3\ln q -8\right\} \leq (q+1) \log_q\Bigl(\frac{\vert D\vert^{1/2}}{h({\mathcal O})}\Bigr) + 10(q+1)\frac{\log_q\vert D\vert}{\log_q\log_q\vert D\vert^{1/2}}.
\end{equation}

It is easy to see that when $\vert D \vert$ is large, then the inequality (\ref{laclef}) cannot hold.
To see this, rewrite (\ref{laclef}) in the following way:
\begin{equation}\label{laclefbis}
\max\left\{\frac{\vert D\vert^{1/2}}{h({\mathcal O})},c_0\log_q\vert D\vert\right\} \leq c_1 \log_q\Bigl(\frac{\vert D\vert^{1/2}}{h({\mathcal O})}\Bigr) + c_2\frac{\log_q\vert D\vert}{\log_q\log_q\vert D\vert^{1/2}},
\end{equation}
where $c_0$, $c_1$ and $c_2$ are positive constants depending on $q$ only. If $\frac{\vert D\vert^{1/2}}{h({\mathcal O})}<c_0\log_q\vert D\vert$, then $\log_q(\frac{\vert D\vert^{1/2}}{h({\mathcal O})})<\log_q c_0+\log_q\log_q\vert D\vert$ and (\ref{laclefbis}) yields
$$
c_0 \log_q\vert D\vert < c_1\log_q c_0+c_1\log_q\log_q\vert D\vert + c_2 \frac{\log_q\vert D\vert}{\log_q\log_q\vert D\vert^{1/2}},
$$
which is impossible if $\vert D \vert > c_q$, where $c_q$ is some sufficiently large constant depending on $q$.
Assume now that
\begin{equation}\label{max-2ndcase}
c_0\log_q\vert D\vert \leq \frac{\vert D\vert^{1/2}}{h({\mathcal O})}.
\end{equation}
Then (\ref{laclefbis}) gives
$$
\frac{\vert D\vert^{1/2}}{h({\mathcal O})}
\leq c_1 \log_q\Bigl(\frac{\vert D\vert^{1/2}}{h({\mathcal O})}\Bigr)
+ \frac{c_2}{c_0 \log_q\log_q\vert D\vert^{1/2}}\, 
\frac{\vert D\vert^{1/2}}{h({\mathcal O})}.
$$
For $\vert D\vert\ge c_4$ with $c_4$ sufficiently large we have $\frac{c_2}{c_0 \log_q\log_q\vert D\vert^{1/2}}\leq \frac{1}{2}$,
hence we obtain
$$
\frac{\vert D\vert^{1/2}}{2h({\mathcal O})}
\leq c_1 \log_q\Bigl(\frac{\vert D\vert^{1/2}}{h({\mathcal O})}\Bigr).
$$
This inequality is impossible if $\vert D\vert^{1/2}/h({\mathcal O})$ is large enough, which is the case if $\vert D \vert> c_q$ for some constant $c_q$ large enough, by (\ref{max-2ndcase}).

Thus we have obtained that there is a constant $c_q>0$ such that $\vert D \vert \leq  c_q$. 
Hence there are finitely many possible values for $D$. It remains to justify that this implies finiteness of the Drinfeld singular moduli. If $q$ is odd, we conclude that there are finitely
many possible values for $\alpha$ for instance by Lemma~\ref{Ord} (2).
If $q$ is even, using the notations $\xi$, $B$, $C$, $G$ of the beginning of Section~\ref{CMcar2}, we see from the
relation $D=f^2G^2$ that there are finitely many possibilities for $f$ and $G$. Then there are finitely many possible values
for $C$ since $G^2= C \, \mathrm{rad}(G)$, and therefore finitely many possible values for $B$ since $\deg B=\deg C$ in the inert case.
It follows from (\ref{xi}) that there are finitely many possibilities for $\xi$. Finally, since the set $S_f(\xi)$ is finite,
it follows from Lemma~\ref{Ord2} (2) that there are finitely many possible values for $\alpha$.

To complete the proof of Theorem~\ref{mainth}, it remains to determine explicitly the constant $c_q$.
Coming back to (\ref{laclef}) (which, we recall, was obtained for $\vert D \vert\ge q^4$), we will see that this inequality implies:
\begin{equation} \label{explicitbound}
\log_q \log_q \sqrt{\vert D \vert} \leq  \frac{2400 (q+1)}{\ln 2}.
\end{equation}

Let us denote $$M=\max\left\{\frac{\sqrt{\vert D\vert}}{h(\mathcal{O})}, \frac{\ln 2}{120}\log_q\vert D\vert -3\ln q -8 \right\}.$$

Let us rewrite inequality (\ref{laclef}):
\begin{equation}\label{star}
M \leq (q+1) \log_q \frac{\sqrt{\vert D\vert}}{h(\mathcal{O})} + 10(q+1)\frac{\log_q\vert D\vert}{\log_q \log_q \sqrt{\vert D\vert}}.
\end{equation}

Let us first assume that $M=\frac{\ln 2}{120}\log_q\vert D\vert -3\ln q -8$. Then by definition of $M$ we have in particular $$\log_q \frac{\sqrt{\vert D\vert}}{h(\mathcal{O})}\leq \log_q \left(\frac{\ln 2}{120}\log_q\vert D\vert -3\ln q -8\right) \leq\log_q \left(\frac{\ln 2}{120}\log_q\vert D\vert\right),$$ hence $$(q+1)\log_q \frac{\sqrt{\vert D\vert}}{h(\mathcal{O})}\leq (q+1)\log_q \left(\frac{\ln 2}{120}\log_q\vert D\vert\right).$$ Furthermore, inequality (\ref{star}) implies 
$$\frac{\ln 2}{120}\log_q\vert D\vert -3\ln q -8\leq (q+1)\log_q \left(\frac{\ln 2}{120}\log_q\vert D\vert \right) + 10(q+1)\frac{\log_q\vert D\vert}{\log_q\log_q \sqrt{\vert D\vert}},$$ which we rearrange as
\begin{equation}\label{intermezzo1}
\left[ \frac{\ln 2}{120} - \frac{10(q+1)}{\log_q\log_q\sqrt{\vert D\vert}} \right] \log_q \vert D\vert -3\ln q -8 \leq (q+1)\log_q \left(\frac{\ln 2}{120}\log_q\vert D\vert \right).
\end{equation}
We now analyse the situation (\ref{intermezzo1}) with two cases: either $$\frac{\ln 2}{120} - \frac{10(q+1)}{\log_q\log_q\sqrt{\vert D\vert}} \leq \frac{1}{2}\cdot \frac{\ln 2}{120},$$ which implies \begin{equation}\label{case1}
\log_q\log_q \sqrt{\vert D\vert}\leq \frac{2400 (q+1)}{\ln 2},
\end{equation}
or we get from (\ref{intermezzo1}) the inequality 
$$\frac{1}{2} \cdot \frac{\ln 2}{120} \log_q\vert D\vert \leq (q+1)\log_q \left(\frac{\ln 2}{120}\log_q\vert D\vert  \right)+3\ln q +8, $$ which gives
$$\log_q\vert D\vert \leq \frac{240(q+1)}{\ln 2}\log_q \left(\frac{\ln 2}{120}\log_q\vert D\vert \right)+\frac{720\ln q+1920}{\ln 2}.$$

The function $g:x\mapsto x-\frac{240(q+1)}{\ln 2}\log_q(\frac{\ln 2}{120} x)-\frac{720\ln q+1920}{\ln 2}$ is increasing on the interval $[\frac{240(q+1)}{\ln2\ln q},+\infty[$, and satisfies $g(N^2)\geq0$ for any $N\geq \frac{240(q+1)}{\ln 2\ln q}$, hence $g(x)\leq 0$ implies in particular $x\leq N^2$. This applied to $N=\frac{240(q+1)}{\ln 2\ln q}$ gives 

\begin{equation}\label{case2}
\log_q\vert D\vert \leq \left(\frac{240(q+1)}{\ln 2 \ln q}\right)^2.
\end{equation}
Back to (\ref{star}), let us now assume that $M=\frac{\sqrt{\vert D\vert}}{h(\mathcal{O})}$.  From (\ref{star}) and the definition of $M$ we get 
$$ \frac{\sqrt{\vert D\vert}}{h(\mathcal{O})}\leq (q+1) \log_q\frac{\sqrt{\vert D\vert}}{h(\mathcal{O})}+\frac{10(q+1)}{\log_q\log_q\sqrt{\vert D \vert}}\left(\frac{\sqrt{\vert D\vert}}{h(\mathcal{O})}+3\ln q +8\right) \frac{120}{\ln 2}, $$ which implies 
\begin{equation}\label{intermezzo2}
\left[1-\frac{1200(q+1)}{\ln 2}\frac{1}{\log_q\log_q \sqrt{\vert D\vert}} \right]\frac{\sqrt{\vert D\vert}}{h(\mathcal{O})} \leq (q+1)\log_q \frac{\sqrt{\vert D\vert}}{h(\mathcal{O})}+\frac{1200(q+1)(3\ln q +8)}{(\log_q\log_q\sqrt{\vert D\vert}) \ln 2}.
\end{equation}
We now analyse the situation (\ref{intermezzo2}) with two cases: either $$1-\frac{1200(q+1)}{\ln 2}\frac{1}{\log_q\log_q \sqrt{\vert D\vert}}\leq \frac{1}{2},$$ which leads to 
\begin{equation}\label{case3}
\log_q\log_q \sqrt{\vert D\vert}\leq \frac{2400(q+1)}{\ln 2},
\end{equation}
or the inequality $(\ref{intermezzo2})$ implies 
\begin{equation}\label{on y va}
    \frac{\sqrt{\vert D\vert}}{2 h(\mathcal{O})}\leq (q+1)\log_q \frac{\sqrt{\vert D\vert}}{h(\mathcal{O})}+\frac{1200(q+1)(3\ln q+8)}{(\log_q\log_q\sqrt{\vert D\vert}) \ln 2},
    \end{equation}
    and one may assume $$\frac{2400(q+1)}{\ln 2}\leq\log_q\log_q \sqrt{\vert D\vert},$$ which gives in (\ref{on y va}) $$\frac{\sqrt{\vert D\vert}}{2 h(\mathcal{O})}\leq (q+1)\log_q \frac{\sqrt{\vert D\vert}}{h(\mathcal{O})}+\frac{3}{2}\ln q +4.$$ For a real parameter $N\geq1$, the continuous function $f:x\to \frac{x}{2}-N\log_q x-\frac{3}{2}\ln q-4$ is increasing on the interval $[\frac{2N}{\ln q},+\infty[$, and satisfies $f(10N^2)\geq0$ when $N\geq q+1$, hence $f(x)\leq 0$ implies in particular $x\leq 10N^2$. This applied to $N=q+1$ gives 
\begin{equation}\label{notenough}
\frac{\sqrt{\vert D\vert}}{h(\mathcal{O})}\leq 10(q+1)^2.
\end{equation}
In view of the Brauer-Siegel theorem, $(\ref{notenough})$ is unfortunately not enough to deduce an upper bound on $\vert D \vert$. We will thus need more information, which will come from the definition of $M$. Applying (\ref{notenough}) leads indeed to 
$$\frac{\ln 2}{120}\log_q\vert D\vert -3\ln q-8 \leq M = \frac{\sqrt{\vert D\vert}}{h(\mathcal{O})}\leq 10(q+1)^2,$$ and we get the upper bound 
\begin{equation}\label{case4}
\log_q\vert D\vert \leq \frac{1200(q+1)^2}{\ln 2}+\frac{120(3\ln q+8)}{\ln2}.
\end{equation}

Compiling the estimates $(\ref{case1}), (\ref{case2}), (\ref{case3}), (\ref{case4})$ leads to the inequality (\ref{explicitbound}).

\end{proof}

\appendix 

\section {Inseparability}\label{inseparability}

The goal of this appendix is to discuss inseparability criterions for values of classical modular forms for Drinfeld modules of rank 2. Such a criterion is used in the CM case in the proof of Lemma \ref{separabilityj-Schweizer}, which in turn is used in the proof of Theorem \ref{Effect1-oddchar-intro}. We give a more general result in this appendix. We keep the notation introduced in the article, recalling some of it when deemed necessary. 
The characteristic of $\BF_q$ will be denoted by $p$.
Recall that $A=\BF_q[T]$ and that we denote by $\bar{k}_\infty$ a fixed algebraic closure of $k_\infty =\mathbb F_q((\frac{1}{T}))$.
Let $k_\infty^{sep}\subset \bar {k}_\infty$ be the separable closure of $k_\infty$ in $\bar{k}_\infty$, and denote
by $v_\infty$ the $\infty$-adic valuation on $\bar{k}_\infty$, normalized such that $v_\infty (T)=-1.$  Thus for all $x\in \bar{k}_\infty,$  we have:
$$\vert x\vert =q^{-v_\infty(x)}.$$
Let $\mathcal{D}=\{z\in \bar{k}_\infty, |z| =|z|_i \geq 1\}$, where we recall that $\vert z \vert_i$ is defined by
$$\vert z\vert _i ={\rm inf}\{ \vert z-\alpha \vert, \, \alpha \in k_\infty \}.$$
We observe that for any $a\in A\setminus\{0\}$, we have $a\mathcal{D}\subset \mathcal{D}.$

Let $z\in \bar{k}_\infty\setminus k_\infty .$ Recall that $\Lambda_z= A\oplus Az$ is a $A$-lattice in $\mathbb C_\infty$ and that there is a Drinfeld module of rank 2 denoted by $\Phi^{(z)}$ associated to $\Lambda_z$ which is of the form:
$$\Phi^{(z)}_T= T+g(z)\tau+\Delta(z) \tau^2,$$ 
where $g(z), \Delta(z)\in \bar{k}_\infty$ with $\Delta(z)\not =0.$
Finally, recall that:
$$j(z) = \frac{g(z)^{q+1}}{\Delta(z)}\in \bar{k}_\infty.$$
Let $e_C(X)$ be the Carlitz exponential:
$$e_C(X)=\sum_{i\geq 0} \frac{1}{D_i} X^{q^i},$$
where $D_0= 1$ and $$\forall i\geq 1, D_i =\prod_{k=0}^{i-1} (T^{q^i}-T^{q^k}).$$
Recall that: ${\rm Ker} (e_C: \mathbb C_\infty \rightarrow \mathbb C_\infty)= \widetilde{\pi} A,$
where  $\widetilde{\pi}$ is the Carlitz period (well-defined modulo $\mathbb F_q^\times$):
$$\widetilde{\pi} = \sqrt[q-1]{-T}\, T \prod_{k\geq 1} (1-T^{1-q^k})^{-1}\in k_\infty^{sep}.$$
For $z\notin A$, let's set:
$$t(z) =\frac{1}{e_C(\widetilde{\pi}z)}\in  k_\infty(\widetilde{\pi}, z)^\times.$$

\begin{lem}\label{Brown} 
Let $z\in D.$ Let $n\geq 0$ be the smallest integer such that $n\geq -v_\infty(z).$ Let's write  $-v_\infty(z) = n-\varepsilon$, where $\varepsilon \in [0, 1[.$ \par
\noindent i) For all $a\in A\setminus\{0\},$ we have:
$$v_\infty(t(az))= \left(\frac{q}{q-1}-\varepsilon\right) q^{n+\deg a}.$$
ii) For $a\in A\setminus \{0\},$ set $\delta_a(z)= \frac{1}{t(az)}- \widetilde{\pi} az.$  Then:
$$v_\infty(\delta_a(z))=-\left(\frac{q}{q-1}-\varepsilon\right)q^{n+\deg a}.$$
\end{lem}
\begin{proof} The assertions are  consequences  of \cite{Brown}, Lemma 2.6.1 and its proof.

\end{proof}

For $m\in \mathbb N,$ we set:
$$A_{+,m}=\{ a\in A, \, a \, {\rm monic}\, , \, \deg a=m\}.$$
As a consequence of the Lemma \ref{Brown}, we get:

\begin{lem} \label{nonvanishing} Let $z \in D.$ Let $n\geq 0$ be the smallest integer such that $n\geq -v_\infty(z).$ Let's write  $-v_\infty(z) = n-\varepsilon, \varepsilon \in [0, 1[.$ Let $\delta, \mu, \nu \in \mathbb N$  such that $\delta \geq \nu+1.$  Then the sum
\begin{equation*}
\sum_{m\geq0}\sum_{a\in A_{+,m}} a^\mu t(az)^{\delta}\delta_a(z)^{\nu}
\end{equation*}
converges in $\bar k_\infty.$ If moreover $ \mu  <q^n(\delta-\nu)(q-\varepsilon (q-1)),$ then:
$$v_\infty\left(\sum_{m\geq0}\sum_{a\in A_{+,m}} a^\mu t(az)^{\delta}\delta_a(z)^{\nu}\right)= (\delta-\nu) \left(\frac{q}{q-1}-\varepsilon\right) q^n.$$
\end{lem}
\begin{proof}  Let's set $\gamma=\delta-\nu \in \mathbb N.$  Let $m\geq 0$ and $a\in A_{+,m}$. Then by Lemma \ref{Brown}:
$$v_\infty(a^{\mu}t(az)^{\delta} \delta_a(z)^{\nu})=  \gamma \left(\frac{q}{q-1}-\varepsilon\right)q^{n+m}-\mu m.$$
Since $ q/(q-1)-\varepsilon >0$ we get the convergence assertion.
Suppose now that $\mu < \gamma q^n(q-\varepsilon (q-1))$.
Noticing that
$$\frac{q}{q-1}-\varepsilon > \frac{1}{q-1},$$ 
we see that for $m\geq 1$ we have
$$\gamma\left(\frac{q}{q-1}-\varepsilon\right)q^n (q^{m}-1)-\mu m>0.$$
Hence for every $m\ge 1$ and $a\in A_{+,m}$ we get
$$ v_\infty(a^\mu t(az)^{\delta}\delta_a(z)^{\nu}) > v_\infty(t(z)^{\delta}\delta_1(z)^{\nu}).$$
Thus:
$$ \sum_{m\geq0}\sum_{a\in A_{+,m}} a^\mu t(az)^{\delta} \delta_a(z)^{\nu}= t(z)^{\delta} \delta_1(z)^{\nu} (1+u)\ {\rm where}\ v_\infty (u)>0.$$
The claim follows.
\end{proof}

We will also need the following basic Lemma in the sequel:\par
\begin{lem}\label{inseparable} 
Let $z\in \bar k_\infty.$ \par
\noindent The following assertions are equivalent:\par
 i) $z\not  \in k_\infty^{sep},$\par
 ii) $[k_\infty(\widetilde{\pi},z):k_\infty(\widetilde{\pi},z^p)]\geq 2,$\par
 iii) $[k_\infty(\widetilde{\pi},z):k_\infty(\widetilde{\pi},z^p)]=p.$\par
\end{lem}
\begin{proof}
Observe that $k_\infty(\widetilde{\pi},z)/k_\infty(\widetilde{\pi},z^p)$ is a purely inseparable extension. Thus if $z\in k_\infty^{sep},$ then $k_\infty(\widetilde{\pi},z)\subset k_\infty ^{sep}$  and therefore  $k_\infty(\widetilde{\pi},z)=k_\infty(\widetilde{\pi},z^p).$\par
\noindent Now let's assume that $z\not \in k_\infty^{sep}.$ We will prove that ii) holds. Otherwise:
$$z\in k_\infty(\widetilde{\pi},z^p).$$
Thus, for all $m\geq 1$:
$$z\in k_\infty(\widetilde{\pi},z^{p^m}).$$
We get a contradiction since there exists an integer $m\geq 1$ such that $z^{p^m}\in k_\infty^{sep}.$ The first assertion follows.\par
\end{proof}

\begin{prop}\label{separabilityj} Let $x\in \bar{k}_\infty\setminus k_\infty.$ The following assertions are equivalent:\par
\noindent i) $x\in k_\infty^{sep},$\par
\noindent ii) $g(x), \Delta(x) \in k_\infty^{sep},$\par
\noindent iii) $j(x) \in k_\infty^{sep}.$
\end{prop}
\begin{proof}  Let's recall that the fonction $j$ induces a bijection:
$$GL_2(A)\backslash (\bar{k}_\infty\setminus k_\infty)\rightarrow \bar{k}_\infty.$$
Let's observe that $\Lambda_x \subset k_\infty(x).$ Let:
$$e^{(x)}(X)= X\prod_{u\in \Lambda_x \setminus\{0\}} \left(1-\frac{X}{u}\right).$$
Then $e^{(x)}\in k_\infty(x) \{\{\tau\}\}.$ Since in $\mathbb C_\infty \{\{ \tau\}\}$ we have
$$\phi^{(x)}_T e^{(x)}= e^{(x)} T,$$
 then $g(x), \Delta(x) \in k_\infty(x).$ Thus i) implies ii). It is obvious that ii) implies iii). Let's prove that iii) implies i).\par
 \noindent  Since $j^{-1}(\{0\})= GL_2(A) (\mathbb F_{q^2}\setminus \mathbb F_q) \subset k_\infty^{sep},$ we can assume that $j:=j(x) \not =0.$ Let $\phi$ be the rank two Drinfeld module given by:
 $$\phi_T=T+\tau +\frac{1}{j} \tau^2\in k_\infty^{sep}\{\tau\}.$$
 Let  $\exp_\phi$ be the exponential function attached to $\phi,$  $\Lambda ={\rm Ker \exp_\phi} = \pi_1A\oplus \pi_2 A,$ and set $z= \frac{\pi_2}{\pi_1} \in \bar {k}_\infty \setminus k_\infty.$ Then:
 $$z\in GL_2(A).x\, .$$
 Now, there exists  a unique element $\log_\phi \in 1+k_\infty^{sep}\{\{\tau\}\}\tau$ such that:
 $$\log_\phi \phi_T= T\log_\phi.$$
 Let $a\in A$ such that $\deg a\gg 0,$ then:
 $$\frac{\pi_i}{a}=\log_\phi \left(\exp_\phi\left(\frac{\pi_i}{a}\right)\right), i=1,2.$$
But observe that $\exp_\phi(\frac{\pi_i}{a})\in k_\infty(j)^{sep}\subset k_\infty^{sep}, , i=1,2.$ Thus : $\pi_1, \pi_2 \in k_\infty(j)^{sep}, i=1,2$ and therefore $z\in k_\infty^{sep}.$ We conclude that $x\in k_\infty^{sep}.$
 \end{proof}

\begin{rem} A special case of the above Proposition was already observed by Schweizer (\cite{Sch97}, Lemma 4).
\end{rem}

It is possible to improve Proposition~\ref{separabilityj} when the element $x$ belongs to the fundamental domain $D$. We do not know
if Proposition~\ref{separabilityjD} remains true when $z\in \bar k_\infty\setminus D$.

\begin{prop}\label{separabilityjD} Let $z\in D.$ The following assertions are equivalent:\par
\noindent i) $z\in k_\infty^{sep},$\par
\noindent ii) $j(z) \in k_\infty^{sep},$\par
\noindent iii) $g(z) \in k_\infty^{sep},$\par
\noindent iv) $\Delta(z) \in k_\infty^{sep}.$
\end{prop}

\begin{proof}
We have already seen in Proposition~\ref{separabilityj} that $i)\Leftrightarrow ii)$. Let us prove that $i)\Leftrightarrow iii)$.
We have:
$$g(z) = \widetilde{\pi}^{q-1} \left(1-(T^q-T)\sum_{a\in A_+}t(az)^{q-1}\right) \in k_\infty(\widetilde{\pi}, z).$$
Thus if $z\in k_\infty^{sep}$ then $ g(z) \in k_\infty^{sep}.$\par
\noindent Let's assume $z\not \in k_\infty^{sep}$ and let's set $F= k_\infty(\widetilde {\pi},z^p).$ Then by Lemma \ref{inseparable}:
$$F(z)=\bigoplus_{k=0}^{p-1} Fz^k.$$
Now observe that:
$$t(az)^{-1} = e_C(\widetilde{\pi} az)= \widetilde{\pi} az+ \delta_a(z), \delta_a(z) \in F.$$
By Lemma \ref{nonvanishing}, we deduce that:
$$g(z) =   \widetilde{\pi}^{q-1} \left(1-\Big(\widetilde{\pi}(T^q-T)\sum_{a\in A_+}at(az)^{q} \Big)z +\xi\right), \xi \in F.$$
Now, again by Lemma \ref{nonvanishing}:
$$\widetilde{\pi}(T^q-T)\sum_{a\in A_+}at(az)^{q} \in F^\times.$$
Thus $g(z)$ writes $g(z)=\xi_1 z +\xi_2$ with $\xi_1\in F^{\times}$ and $\xi_2\in F$. It follows that $g(z)\notin F$, hence $g(z)\not \in k_\infty (\widetilde{\pi}, g(z)^p)$
since $k_\infty (\widetilde{\pi}, g(z)^p)\subset F$.
We conclude by Lemma~\ref{inseparable} that  $g(z)\not \in k_\infty^{sep}.$

Let us now prove the equivalence $i)\Leftrightarrow iv)$. We already know that $i) \Rightarrow iv)$, so let us prove the converse.
By \cite{Lop2010}, we have:
$$\widetilde{\pi} ^{1-q^2}\Delta(z)=  -\sum_{a\in A_+} a^{q(q-1)} t(az)^{q-1}.$$
Let's assume that $ z\not \in k_\infty^{sep}$ and let $F=k_\infty(\widetilde{\pi}, z^p).$ Then, by Lemma \ref{nonvanishing}, and in the same way as before we deduce that:
$$\widetilde{\pi} ^{1-q^2}\Delta(z)=\left(-\widetilde{\pi}\sum_{a\in A_+} a^{q(q-1)+1} t(az)^{q}\right) z +\xi , \quad \xi \in F.$$
Now, by Lemma \ref{nonvanishing}, we get:
$$\sum_{a\in A_+} a^{q(q-1)+1} t(az)^{q}\not =0$$
(note that in Lemma \ref{nonvanishing}, if $n=0$ then $\varepsilon=0$ since $z\in D$).
It follows that $\Delta(z) \notin k_\infty^{sep}$ as before.
\end{proof}

\end{document}